\documentclass[a4paper, 11pt]{amsart}
\usepackage{graphicx, euscript, enumitem,kantlipsum}
\usepackage{amssymb}
\usepackage{amsmath}
\usepackage{amsthm}
\usepackage{amscd}
\usepackage{mathbbol}
\usepackage[all,2cell]{xy}

\usepackage[pagebackref,colorlinks]{hyperref}

\usepackage{tikz-cd}
\usetikzlibrary{arrows,backgrounds,shadows,decorations.markings,matrix}

\usepackage{geometry}
\UseAllTwocells \SilentMatrices
\newtheorem{thm}{Theorem}[section]

\newtheorem{cor}[thm]{Corollary}
\newtheorem{lem}[thm]{Lemma}
\newtheorem{exm}[thm]{Example}

\newtheorem{thmy}{Theorem}

\newtheorem{prop}[thm]{Proposition}
\theoremstyle{definition}
\newtheorem{defn}[thm]{Definition}
\theoremstyle{remark}
\newtheorem{rem}[thm]{\bf Remark}
\numberwithin{equation}{section}

\makeatletter
\newcommand{\smallotimes}{\mathbin{\mathpalette\make@small\otimes}}

\newcommand{\make@small}[2]{%
  \vcenter{\hbox{%
    $\m@th\ifx#1\displaystyle\scriptstyle\else\ifx#1\textstyle\scriptstyle
     \else\scriptscriptstyle\fi\fi#2$%
  }}%
}
\makeatother

\begin{document}

\title[Module factorizations]{Module factorizations}

\author{Xiao-Wu Chen}

\makeatletter
\@namedef{subjclassname@2020}{\textup{2020} Mathematics Subject Classification}
\makeatother


\date{\today}

\begin{abstract}
For a regular  normal element in an arbitrary ring, we study the category of its module factorizations. The cokernel functor  relates module factorizations with Gorenstein projective  components to  Gorenstein projective modules over the quotient ring. The results are vast extensions of Eisenbud's matrix factorization theorem.
\end{abstract}

\maketitle


\section{Introduction}

Matrix factorizations are introduced in \cite{Eis}, which are used to obtain periodic free resolutions over hypersurface singularities. They have important applications in knot theory \cite{KR}, string theory \cite{KapL, KKP, Car}, Hodge theory \cite{BFK}, singularity categories \cite{Buc,Orl} and representation theory \cite{KST} of quivers. We refer to \cite{Nee} for a leisurely introduction.

In what follows, we will first recall three theorems relating matrix factorizations to finitely generated Gorenstein projective modules over the quotient ring: the original theorem \cite{Eis} by Eisenbud, its global version \cite{Orl} and its noncommutative version \cite{CCKM, MU}.  Then we describe the main results, which relate module factorizations with Gorenstein projective components to Gorenstein projective modules over the quotient ring. The key feature is that we deal with arbitrary modules over arbitrary rings. Furthermore, we  strengthen the noncommutative version of Eisenbud's theorem.

\subsection{Matrix factorization theorems}

Let $S$ be a commutative local noetherian ring with $\omega \in S$. A \emph{matrix factorization} of $\omega$ over $S$ is a quadruple $X=(X^0, X^1; d_X^0, d_X^1)$ with each $X^i$ a finitely generated free $S$-module,  $d_X^0\colon X^0\rightarrow X^1$ and $d_X^1\colon X^1\rightarrow X^0$ two morphisms of $S$-modules satisfying $d_X^1\circ d_X^0=\omega {\rm Id}_{X^0}$ and $d_X^0\circ d_X^1=\omega {\rm Id}_{X^1}$. If we choose bases for $X^i$,  the two morphisms $d_X^i$ are represented by two matrices $D^i$ with entries in $S$. Then we have the following genuine factorizations of the scalar matrix $\omega I$:
$$D^1D^0=\omega I=D^0D^1.$$
Here, $I$ denotes the identity matrix. Denote by $\mathbf{MF}(S; \omega)$ the category formed by matrix factorizations, and by $\underline{\mathbf{MF}}(S; \omega)$ the stable category modulo null-homotopical morphisms.

Consider the quotient ring $\bar{S}=S/{(\omega)}$. Denote by $\bar{S}\mbox{-mod}$ the abelian category of finitely generated $\bar{S}$-modules,  and by $\bar{S}\mbox{-}\underline{\rm mod}$ the stable category modulo projective modules. Denote by ${\rm MCM}(\bar{S})$ the full subcategory of $\bar{S}\mbox{-mod}$ formed by maximal Cohen-Macaulay $\bar{S}$-modules and by $\underline{\rm MCM}(\bar{S})$ the corresponding stable category.

The zeroth cokernel functor
$${\rm Cok}^0\colon \mathbf{MF}(S; \omega)\longrightarrow \bar{S}\mbox{-mod}$$
sends any matrix factorization $X=(X^0, X^1; d_X^0, d_X^1)$ to  the cokernel of $d_X^0\colon X^0\rightarrow X^1$, which is denoted by ${\rm Cok}^0(X)$. Since $\omega$ vanishes on ${\rm Cok}^0(X)$, it is naturally an $\bar{S}$-module.

The following matrix factorization theorem is classical; see \cite[Section~6]{Eis}.
\vskip 5pt

\noindent {\bf Theorem~A.} (Eisenbud) {\emph  Assume that $S$ is a regular local ring and  that $\omega\in S$ is nonzero. Then the zeroth  cokernel functor induces an equivalence
$${\rm Cok}^0\colon \underline{\mathbf{MF}}(S; \omega)\stackrel{\sim}\longrightarrow \underline{\rm MCM}(\bar{S}).$$
}

\vskip 5pt

The stable category $\underline{\mathbf{MF}}(S; \omega)$ of matrix factorizations is always canonically triangulated. In the setup above, the quotient ring $\bar{S}$ is Gorenstein. Then $\underline{\rm MCM}(\bar{S})$ is also canonically triangulated \cite{Hap}. The above equivalence is actually a triangle equivalence. We emphasize the necessity of the condition that $\omega$ is nonzero.

To describe a global version of Theorem~A, we assume that $S$ is a commutative noetherian ring which is not necessarily local. Recall that an element $\omega\in S$ is called \emph{regular} if it is a  non-zero-divisor.

Fix a regular element $\omega\in S$. We have to extend the definition of a matrix factorization of $\omega$ over $S$ by allowing each component $X^i$ to be a finitely generated \emph{projective} $S$-module. Denote by $\bar{S}\mbox{-Gproj}$ the full subcategory of $\bar{S}\mbox{-mod}$ formed by finitely generated \emph{Gorenstein projective $\bar{S}$-modules} \cite{ABr}, and by $\bar{S}\mbox{-}\underline{\rm Gproj}$ its stable category. Finitely generated Gorenstein projective modules are called modules of G-dimension zero in \cite{ABr}, Cohen-Macaulay modules in \cite{Bel00} and totally reflexive modules in \cite{AM}.

When $S$ is regular local, any nonzero element $\omega$ is regular. The  quotient ring $\bar{S}$ is local Gorenstein, and  we have ${\rm MCM}(\bar{S})=\bar{S}\mbox{-Gproj}$; see \cite[(4.13)~Theorem]{ABr}. This is one of the main motivations to study finitely generated Gorenstein projective modules.

The following global version of Theorem~A  is a slight reformulation of  \cite[Theorem~3.9]{Orl}. We mention that its non-affine version is given in \cite[Theorem~2]{Orl2}; compare \cite{PV, EfP, BW}.

\vskip 5pt

\noindent {\bf Theorem~B.}  {\emph  Assume that $S$ is a commutative noetherian ring with finite global dimension and  that $\omega\in S$ is regular. Then the zeroth cokernel functor induces a triangle equivalence
$${\rm Cok}^0\colon \underline{\mathbf{MF}}(S; \omega)\stackrel{\sim}\longrightarrow \bar{S}\mbox{-}\underline{\rm Gproj}.$$
}

\vskip 5pt
We will recall a noncommutative version of Theorem~B from \cite{CCKM, MU}. Let $A$ be a left noetherian ring. An element $\omega\in A$ is \emph{regular} if $a\omega=0$ or $\omega a=0$ implies that $a=0$. An element $\omega\in A$ is \emph{normal} if $A\omega=\omega A$; in this case, both equals $(\omega)=A\omega A$. Then there is a unique ring automorphism $\sigma\colon A\rightarrow A$ such that
$\omega a=\sigma(a)\omega$ holds for each $a\in A$.  When the regular element $\omega$ is central, we have $\sigma={\rm Id}_A$. Regular normal elements play a central role in defining  noncommutative complete intersections  \cite{KKZ} in  noncommutative algebraic geometry.

For each left $A$-module $M$, denote by ${^\sigma(M)}$ its \emph{twisted module}. Its typical element is denoted by ${^\sigma(m)}$, and its left $A$-action is given by $a{^\sigma(m)}={^\sigma(\sigma(a)m)}$. We have a canonical morphism  of $A$-modules
$$\omega_M\colon M\longrightarrow {^\sigma(M)},$$
 which sends $m$ to ${^\sigma(\omega m)}$.

By a \emph{matrix factorization} \cite{CCKM} of $\omega$ over $A$, we mean a quadruple $X=(X^0, X^1; d_X^0, d_X^1)$ such that each $X^i$ is a finitely generated projective $A$-module, $d^0_X\colon X^0\rightarrow X^1$ and $d_X^1\colon X^1\rightarrow {^\sigma(X^0)}$ are morphisms satisfying
$$d_X^1\circ d_X^0=\omega_{X^0} \mbox{ and } {^\sigma(d_X^0)}\circ d_X^1=\omega_{X^1}.$$
We denote by $\mathbf{MF}(A; \omega)$ the category of matrix factorizations of $\omega$, and by $\underline{\mathbf{MF}}(A; \omega)$ its stable category.

Let $\bar{A}=A/{(\omega)}$ be the quotient ring.  We still have the zeroth cokernel functor
$${\rm Cok}^0\colon \mathbf{MF}(A; \omega)\longrightarrow \bar{A}\mbox{-mod},$$
where the $\bar{A}$-module ${\rm Cok}^0(X)$ is defined to be the cokernel of $d_X^0\colon X^0\rightarrow X^1$.

\vskip 5pt

\noindent {\bf Theorem~C.} {\emph  Assume that $A$ is a left noetherian ring with finite left global dimension and  that $\omega\in A$ is regular and normal. Then the zeroth cokernel functor induces a triangle equivalence
$${\rm Cok}^0\colon \underline{\mathbf{MF}}(A; \omega)\stackrel{\sim}\longrightarrow \bar{A}\mbox{-}\underline{\rm Gproj}.$$
}
\vskip 5pt

We mention that a graded version of Theorem~C for AS-regular algebras is due to \cite[Theorem~5.8]{CCKM} and \cite[Theorem~6.6]{MU}. The argument given there applies well to the ungraded case.

\subsection{The  main results}
In this subsection, we will describe the main results.

Throughout the paper, we fix an arbitrary ring $A$ and an element $\omega\in A$ which is regular and normal. There is a unique ring automorphism $\sigma\colon A\rightarrow A$ satisfying $\omega a=\sigma(a)\omega$ for any $a\in A$. Set $\bar{A}=A/(\omega)$ to be the quotient ring.

By a \emph{module factorization} of $\omega$ over $A$, we mean a quadruple $X=(X^0, X^1; d_X^0, d_X^1)$, where both $X^i$ are arbitrary $A$-modules, $d_X^0\colon X^0\rightarrow X^1$ and $d_X^1\colon X^1\rightarrow {^\sigma(X^0)}$ are morphisms satisfying
$$d_X^1\circ d_X^0=\omega_{X^0} \mbox{ and } {^\sigma(d_X^0)}\circ d_X^1=\omega_{X^1}.$$
Denote by $\mathbf{F}(A; \omega)$ the abelian category formed by module factorizations; compare \cite{BE}. Module factorizations over commutative rings are defined in \cite{DM}, which are called \emph{linear factorizations}; see also \cite{BT} and compare \cite[Section~9]{EP}.  We refer to \cite{BDFI, BJ} for  factorization categories, and \cite{BRTV, Pi1, Pi2} for differential graded aspects of matrix factorizations.

Therefore,  a matrix factorization of $\omega$ is just a module factorization $X$ with each component $X^i$ a finitely generated projective $A$-module. They form a full subcategory $\mathbf{MF}(A; \omega)$ of $\mathbf{F}(A; \omega)$. Denote by $\bar{A}\mbox{-Mod}$ the category of left $\bar{A}$-modules. We will consider the \emph{zeroth cokernel functor}
$${\rm Cok}^0\colon\mathbf{F}(A; \omega)\longrightarrow \bar{A}\mbox{-Mod},$$
which sends any module factorization $X$  to the $\bar{A}$-module ${\rm Cok}^0(X)$, that is, the cokernel of $d_X^0\colon X^0\rightarrow X^1$.

We introduce  a brand new homotopy relation, that is, \emph{p-null-homotopical morphisms} between module factorizations in Section~4. For morphisms between matrix factorizations, p-null-homotopical morphisms coincide with the usual null-homotopical morphisms. Denote by $\underline{\mathbf{F}}(A; \omega)$ the stable category of $\mathbf{F}(A; \omega)$ modulo p-null-homotopical morphisms. Then the stable category  $\underline{\mathbf{MF}}(A; \omega)$ is viewed as a full subcategory of $\underline{\mathbf{F}}(A; \omega)$.

Recall that \emph{Gorenstein projective modules} that are not necessarily finitely generated are introduced in \cite{EJ}. These modules play a central role in Gorenstein homological algebras. Gorenstein projective modules form a Frobenius exact category; see \cite[Proposition~3.8]{Bel}. On the other hand, any Frobenius exact category is equivalent to a certain category formed by Gorenstein projective modules over some category; see \cite[Theorem~4.2]{Chen12}. This close link to Frobenius exact categories makes Gorenstein projective modules indispensable in modern homological algebra. Moreover,  the integral representation theory of finite groups \cite{Buc, Bar, BBIKP} mainly studies Gorenstein projective modules over the group rings.

We denote by  $\bar{A}\mbox{-GProj}$ the full subcategory of $\bar{A}\mbox{-Mod}$ formed by Gorenstein projective $\bar{A}$-modules. It stable category $\bar{A}\mbox{-}\underline{\rm GProj}$ modulo projective $\bar{A}$-modules is canonically  triangulated.

Denote by $\mathbf{GF}(A; \omega)$ the full subcategory of $\underline{\mathbf{F}}(A; \omega)$  formed by those module factorizations $X$ with each component $X^i$ a Gorenstein projective $A$-module, and by $\underline{\mathbf{GF}}(A; \omega)$ the corresponding full subcategory of $\underline{\mathbf{F}}(A; \omega)$. It seems to be  nontrivial that  $\mathbf{GF}(A; \omega)$ is naturally a Frobenius exact category, whose proof relies on certain \emph{Frobenius functors} \cite{Mor65}; see Section~3. Consequently, the stable category $\underline{\mathbf{GF}}(A; \omega)$ is canonically triangulated.

Denote by  $\mathbf{G}^0\mathbf{F}(A; \omega)$ the full subcategory of $\mathbf{GF}(A; \omega)$ formed by those module factorizations $X$ with the component $X^1$  projective. The corresponding full subcategory  $\underline{\mathbf{G}^0\mathbf{F}}(A; \omega)$ of $\underline{\mathbf{GF}}(A; \omega)$ is a triangulated subcategory.

\vskip 5pt
\noindent {\bf Theorem~1 (= Theorem~\ref{thm:cok0G}).}  \emph{The zeroth cokernel functor ${\rm Cok}^0\colon\mathbf{F}(A; \omega)\rightarrow \bar{A}\mbox{-}{\rm Mod}$ induces a triangle equivalence
$${\rm Cok}^0\colon \underline{\mathbf{G}^0\mathbf{F}}(A; \omega) \stackrel{\sim}\longrightarrow \bar{A}\mbox{-}\underline{\rm GProj}.$$}
\vskip 5pt

We mention that module factorizations  with maximal Cohen-Macaulay components over a commutative local Gorenstein ring appear implicitly in the study \cite{EP} of layered  resolutions. In view of \cite[Section~9]{EP},  Theorem~1  to some extent, at least its finite commutative version, is expected by experts.

Each $A$-module $M$ yields a \emph{trivial} module factorization $\theta^0(M)=(M, M; {\rm Id}_M, \omega_M)$. Denote by $\mathcal{N}$ the full subcategory of $\underline{\mathbf{GF}}(A; \omega) $ formed by those objects that are isomorphic to $\theta^0(G)$ for some Gorenstein projective $A$-module $G$. It is a triangulated subcategory; moreover, it is triangle equivalent to $A\mbox{-}\underline{\rm GProj}$.

\vskip 5pt
\noindent {\bf Theorem~2 (= Theorem~\ref{thm:quotient}).}  \emph{The zeroth cokernel functor ${\rm Cok}^0\colon\mathbf{F}(A; \omega)\rightarrow \bar{A}\mbox{-}{\rm Mod}$ induces a triangle equivalence
$$ \underline{\mathbf{GF}}(A; \omega)/\mathcal{N} \stackrel{\sim}\longrightarrow \bar{A}\mbox{-}\underline{\rm GProj},$$
where $\underline{\mathbf{GF}}(A; \omega)/\mathcal{N}$ denotes the Verdier quotient category. }
\vskip 5pt

When the ring $A$ has finite left global dimension, any Gorenstein projective $A$-module is projective. Therefore,  the category $\mathcal{N}$ is trivial and $\underline{\mathbf{G}^0\mathbf{F}}(A; \omega)=\underline{\mathbf{GF}}(A; \omega)$. Consequently, the equivalences in Theorems~1 and 2 coincide.

Denote by $\bar{A}\mbox{-}\underline{\rm Gproj}^{<+\infty}$ the full category of $\bar{A}\mbox{-}\underline{\rm GProj}$ formed by those \emph{totally reflexive $\bar{A}$-modules}  with finite projective dimension as $A$-modules; it is  a triangulated subcategory of $\bar{A}\mbox{-}\underline{\rm GProj}$. We mention that if $\bar{A}$ is left coherent,  then totally reflexive $\bar{A}$-modules coincide with finitely presented Gorenstein projective $\bar{A}$-modules.

\vskip 5pt
\noindent {\bf Theorem~3 (= Theorem~\ref{thm:finite}).}  \emph{The equivalence in Theorem~1 restricts to  a triangle equivalence
$$ {\rm Cok}^0\colon \underline{\mathbf{MF}}(A; \omega) \stackrel{\sim}\longrightarrow \bar{A}\mbox{-}\underline{\rm Gproj}^{<+\infty}.$$ }
\vskip 5pt

In Theorem~3, if $A$ has finite left global dimension, we have $\bar{A}\mbox{-}\underline{\rm Gproj}=\bar{A}\mbox{-}\underline{\rm Gproj}^{<+\infty}$. Then we strength Theorem~C.  The equivalence in Theorem~1 is an extension of the one in Theorem~3. Therefore, Theorems~1 and 2 might be viewed as  vast extensions of  Theorem~A, namely Eisenbud's matrix factorization theorem. Even when the ring $A$ is commutative local noetherian, the obtained equivalences in Theorem~1 and~2 seem to be brand new.

Since duality functors do not behave well on arbitrary modules, the argument \cite{CCKM, MU} in proving Theorem~C does not extend to our setting.  The main ingredients of the proof of Theorem~1 are Theorems~\ref{thm:GP} and~\ref{thm:cok0}. Theorem~\ref{thm:GP} implies that an $\bar{A}$-module is Gorenstein projective if and only if its first $A$-syzygy is Gorenstein projective. Theorem~\ref{thm:cok0} establishes an equivalence between $\bar{A}\mbox{-Mod}$ with a certain full subcategory of the stable category of $\mathbf{F}(A; \omega)$. Theorem~2 relies on Theorem~1 and a recollement in Proposition~\ref{prop:recGF}. In a certain sense, the recollement in Proposition~\ref{prop:recGF} illustrates better the relation among $\underline{\mathbf{GF}}(A; \omega)$, $\bar{A}\mbox{-}\underline{\rm GProj}$ and $A\mbox{-}\underline{\rm GProj}$.

The paper is organized as follows. In Section~2, we relate Gorenstein projective  $A$-modules to Gorenstein projective $\bar{A}$-modules in Theorem~\ref{thm:GP}. We prove a change-of-rings theorem on Gorenstein projective dimensions in Theorem~\ref{thm:Gpd}. In Section~3, we relate the category of module factorizations to  the module category over a matrix ring  in Proposition~\ref{prop:F-Gamma}. We obtain various Frobenius functors, and infer that the category $\mathbf{GF}(A; \omega)$ is Frobenius exact in Proposition~\ref{prop:GP}. We introduce p-null-homotopical morphisms in Section~4. We study  the zeroth cokernel functor  and prove Theorem~1 in Section~5. We obtain a recollement in Proposition~\ref{prop:recGF}, and prove Theorem~2 in Section~6. We study matrix factorizations and prove Theorem~3 in Section~7. In the final section, we illustrate the obtained results using integral representations of finite groups.

In what follows, modules are always left modules. We refer to \cite{Hap, Kel} for Frobenius exact categories and triangulated categories. For Gorenstein projective modules, we refer to \cite{EJ, Holm, Christ}.

\section{Gorenstein projective modules and dimensions}

In this section, we study Gorenstein projective modules over $A$ and the quotient ring $\bar{A}=A/{(\omega)}$. The main result implies that an $\bar{A}$-module $N$ is Gorenstein projective if and only if its first syzygy $\Omega_A(N)$ is a Gorenstein projective $A$-module; see Theorem~\ref{thm:GP}. We reformulate Theorem~\ref{thm:GP} as a change-of-rings theorem on Gorenstein projective dimensions in Theorem~\ref{thm:Gpd}, which in turn  allows us to slightly strengthen Theorem~\ref{thm:GP}; see Theorem~\ref{thm:GP-2}.

Denote by  $A\mbox{-Mod}$ the category of left $A$-modules, and by $A\mbox{-Proj}$ its full subcategory formed by projective $A$-modules. The stable category of $A\mbox{-Mod}$ modulo morphisms factoring through projective modules is denoted by $A\mbox{-}\underline{\rm Mod}$. We have the syzygy endofunctor $\Omega_A\colon A\mbox{-}\underline{\rm Mod}\rightarrow A\mbox{-}\underline{\rm Mod}$.

\subsection{Gorenstein projective modules and approximations}

We usually  denote a cochain complex of $A$-modules by $P^\bullet=(P^n, d_P^n)_{n\in \mathbb{Z}}$, where the differentials $d_P^n\colon P^n\rightarrow P^{n+1}$ satisfy $d_P^{n+1}\circ d_P^n=0$. Recall that a unbounded complex $P^\bullet$ is called \emph{totally acyclic} provided that it is acyclic with each component $P^n$ projective and that for each projective $A$-module $Q$, the Hom complex ${\rm Hom}_A(P^\bullet, Q)$ is always acyclic. An $A$-module $L$ is \emph{Gorenstein projective} if there exists a totally acyclic complex $P^\bullet$ with $L\simeq Z^0(P^\bullet)$. Here, $Z^0(P^\bullet)={\rm Ker}(d_P^0)$ is the zeroth cocycle of $P^\bullet$, and the complex $P^\bullet$ is called a \emph{complete resolution} of $L$. We denote by $A\mbox{-GProj}$ the full subcategory formed by Gorenstein projective $A$-modules, and have $A\mbox{-Proj}\subseteq A\mbox{-GProj}$.

It is well known that  $A\mbox{-GProj}$ is closed under extensions in $A\mbox{-Mod}$; see \cite[Theorem~2.5]{Holm} and compare  \cite[Lemma~2.3]{AM}. Therefore, it is naturally an exact category in the sense of \cite{Qui}. Moreover, it is Frobenius, whose projective-injective objects are precisely all projective $A$-modules; see \cite[Proposition~3.8(i)]{Bel}. Consequently, by \cite[I.2]{Hap} the stable category $A\mbox{-}\underline{\rm GProj}$ is naturally triangulated. In particular, the restricted syzygy endofunctor $\Omega_A\colon A\mbox{-}\underline{\rm GProj}\rightarrow A\mbox{-}\underline{\rm GProj}$ is an autoequivalence, and its quasi-inverse is the suspension functor of  $A\mbox{-}\underline{\rm GProj}$.

Let  $\mathcal{S}$ be any full subcategory of $A\mbox{-Mod}$. For an $A$-module $L$, a \emph{left $\mathcal{S}$-approximation} of $L$ means a morphism $f\colon L\rightarrow S$ with $S\in \mathcal{S}$ such that any morphism $t\colon L\rightarrow T$ with $T\in \mathcal{S}$ factors through $f$, that is, there exists $t'\colon S\rightarrow T$ in $\mathcal{S}$ satisfying $t=t'\circ f$. For each $n\geq 1$, we define the following full subcategory
$${^{\perp_{[1, n]}}\mathcal{S}}=\{L\in A\mbox{-}{\rm Mod}\; |\; {\rm Ext}_A^i(L, S)=0 \mbox{ for all } S\in \mathcal{S}, 1\leq i\leq n\}.$$
We set ${^\perp\mathcal{S}}=\bigcap_{n\geq 1} {^{\perp_{[1, n]}}\mathcal{S}}$.

The following fact is standard.

\begin{lem}\label{lem:ext1}
Suppose that $0\rightarrow L' \stackrel{a}\rightarrow P\rightarrow L\rightarrow 0$ is a short exact sequence of $A$-modules with $P$ projective. Then $L$ belongs to ${^{\perp_{1}}(A\mbox{-}{\rm Proj})}$ if and only if $a$ is a left $(A\mbox{-}{\rm Proj})$-approximation of $L'$. \hfill $\square$
\end{lem}

The following well-known result will be useful; see \cite[Proposition~2.3]{Holm}.

\begin{lem}\label{lem:GP}
Let $L$ be an $A$-module. Then $L$ is Gorenstein projective if and only if for each $i\geq 0$, there exists a short exact sequence $0\rightarrow L^i\rightarrow P^i\rightarrow L^{i+1}\rightarrow 0$ of $A$-modules such that each $P^i$ is projective, each $L^i$ belongs to $^\perp(A\mbox{-}{\rm Proj})$ and $L=L^0$. \hfill $\square$
\end{lem}

Recall the twist autoequivalence $(-)^\sigma\colon A\mbox{-Mod}\rightarrow  A\mbox{-Mod}$. For an $A$-module $M$, a typical element of ${^\sigma(M)}$ is written as ${^\sigma(m)}$ and its $A$-action is given such that $a{^\sigma(m)}={^\sigma(\sigma(a)m)}$. For a morphism $f\colon M\rightarrow M'$, the corresponding morphism ${^\sigma(f)}\colon {^\sigma(M)}\rightarrow {^\sigma(M')}$ sends ${^\sigma(m)}$ to ${^\sigma(f(m))}$.

The \emph{canonical morphism}
$$\omega_M\colon M\longrightarrow {^\sigma(M)}$$
sends $m$ to ${^\sigma(\omega m)}$. The $A$-module $M$ is called \emph{$\omega$-torsionfree} if $\omega_M$ is mono, in which case the given element $\omega$ is also called $M$-regular.  Since $\omega$ is regular in $A$, any projective $A$-module is $\omega$-torsionfree.

The following result is elementary.

\begin{lem}\label{lem:elem}
Assume that $M$ and $N$ are $A$-modules satisfying that $M$ is $\omega$-torsionfree and $\omega_N=0$. Then we have ${\rm Hom}_A(N, M)=0$.\hfill $\square$
\end{lem}

Recall that $\bar{A}=A/{(\omega)}$. In what follows, we will always identify any $A$-module $N$ satisfying $\omega_N=0$ with the corresponding $\bar{A}$-module, still denoted by  $N$.

\begin{lem}\label{lem:factor}
Assume that $0\rightarrow L\stackrel{a}\rightarrow M\stackrel{b}\rightarrow N\rightarrow 0$ is a short exact sequence of $A$-modules with $\omega_N=0$. Let $f\colon L\rightarrow T$ be any morphism of $A$-modules. Then there exists a morphism $h\colon M\rightarrow {^\sigma(T)}$ satisfying $h\circ a=\omega_T\circ f$.
\end{lem}

\begin{proof}
By $\omega_T\circ f={^\sigma(f)}\circ \omega_L$, it suffices to prove that $\omega_L$ factors through $a$. Consider the following commutative diagram with exact rows.
\[
\xymatrix{
0\ar[r] & L\ar[d]_-{\omega_L}  \ar[r]^-{a} & M\ar[d]^-{\omega_M} \ar[r]^-{b} & N\ar[d]^-{\omega_N} \ar[r] & 0\\
0\ar[r] & {^\sigma(L)} \ar[r]^-{^\sigma(a)} & {^\sigma(M)} \ar[r]^-{^\sigma(b)} & {^\sigma(N)} \ar[r] & 0
}\]
Since $\omega_N=0$, by a diagram-chasing we deduce the required factorization.
\end{proof}

\begin{prop}\label{prop:approx}
Suppose that there is a commutative diagram of $A$-modules with exact rows.
\[\xymatrix{
0\ar[r] &L\ar[d]_-{i} \ar[r]^-{a} & P \ar[d]^-{j}\ar[r]^-{b} & N \ar[d]^-{k} \ar[r] & 0\\
0\ar[r] &U \ar[r]^-{c} \ar[r] & Q \ar[r]^-{d} & E \ar[r] &0
}\]
Assume that both $P$ and $Q$ are projective, $\omega_N=0=\omega_E$ and that $E$ is projective as an $\bar{A}$-module. Then the following statements hold.
\begin{enumerate}
\item If $i$ is a left $(A\mbox{-}{\rm Proj})$-approximation, then $k$ is a left $(\bar{A}\mbox{-}{\rm Proj})$-approximation.
\item Assume further that $j$ is a split monomorphism. If  $k$ is a left $(\bar{A}\mbox{-}{\rm Proj})$-approximation,  then $i$ is a left $(A\mbox{-}{\rm Proj})$-approximation.
\end{enumerate}
\end{prop}

\begin{proof}
The $\bar{A}$-projectivity of $E$ implies that $E$ has projective dimension one as an $A$-module. Then we infer that $U$ is a projective $A$-module.

To prove (1), we take an arbitrary morphism $t\colon N\rightarrow F$ with $F$ a free $\bar{A}$-module. It suffices to show that $t$ factors through $k$. We have a free $A$-module $V$ such that $F$ fits into a short exact sequence.
\begin{align}\label{equ:VF}
0\longrightarrow V \stackrel{\omega_V} \longrightarrow {^\sigma(V)} \stackrel{p}\longrightarrow F\longrightarrow 0
\end{align}
We use the projectivity of $P$ to obtain the morphisms $r$ and $s$ in the following diagram, which satisfy $t\circ b=p\circ s$ and $s\circ a=\omega_V\circ r$.
\begin{align}\label{cd:TCM}
\xymatrix{
 &0\ar[r] &L \ar[ldd]_>>>>>{r} \ar[d]_-{i} \ar[r]^-{a} & P\ar[ldd]_>>>>>{s} \ar[d]^-{j}\ar[r]^-{b} & N \ar[ldd]_>>>>>{t} \ar[d]^-{k} \ar[r] & 0\\
  & 0\ar[r] &U \ar@{.>}[ld]|{x} \ar[r]^>>>>>>>{c} \ar[r] & Q \ar@{.>}[ld]|{y} \ar[r]^>>>>>>{d} & E \ar@{.>}[ld]|{z} \ar[r] &0\\
0\ar[r] & V \ar[r]_-{\omega_V} & {^\sigma(V)} \ar[r]_-{p} & F \ar[r] & 0
}
\end{align}
Since $i$ is a left $(A\mbox{-Proj})$-approximation, there exists $x\colon U\rightarrow V$ satisfying $r=x\circ i$. By $\omega_E=0$ and  Lemma~\ref{lem:factor}, we have a morphism $y\colon Q\rightarrow {^\sigma(V)}$ satisfying $\omega_V\circ x=y\circ c$. Consequently, we have $z\colon E\rightarrow F$ satisfying $p\circ y=z\circ d$. By a diagram-chasing, we have $(s-y\circ j)\circ a=0$. Therefore, there exists $e\colon N\rightarrow {^\sigma(V)}$ satisfying
$$s-y\circ j=e\circ b.$$
 Since $\omega_N=0$ and ${^\sigma(V)}$ is free, we infer from Lemma~\ref{lem:elem} that $e=0$. Consequently, we have $s=y\circ j$. It follows that $t=z\circ k$, as required.

To prove (2), we take an arbitrary morphism $r\colon L\rightarrow V$ with $V$ a free $A$-module. We will obtain the same diagram (\ref{cd:TCM}). Indeed,  we define the free $\bar{A}$-module $F$ using (\ref{equ:VF}). Applying the condition $\omega_N=0$ and Lemma~\ref{lem:factor}, we still obtain the two morphisms $s$ and $t$. Since $k$ is a left $(\bar{A}\mbox{-Proj})$-approximation, we obtain a morphism $z\colon E\rightarrow F$ satisfying $t=z\circ k$. Since $Q$ is projective, we obtain $x$ and $y$ making the bottom diagram commute.

The subtle fact is  as follows: in general,  the commutativity of the rightmost triangle $NEF$ does not imply the commutativity of the middle  triangle $PQ{^\sigma(V)}$. Instead, by a diagram-chasing we have $p\circ (s-y\circ j)=0$. Consequently, there exists $h\colon P\rightarrow V$ such that
$$s-y\circ j=\omega_V\circ h.$$
Another diagram-chasing yields $r-x\circ i=h\circ a$. Since $j\colon P\rightarrow Q$ is a split monomorphism, we have a morphism $h'\colon Q\rightarrow V$ with the property $h=h'\circ j$. Then we have
$$r-x\circ i=h'\circ j\circ a=h'\circ c\circ i.$$
We obtain  $r=(x+h'\circ c)\circ i$. Consequently, the arbitrarily given morphism $r$ factors through $i$. It follows that $i$ is a left $(A\mbox{-Proj})$-approximation.
\end{proof}

\begin{prop}
Assume that  $0\rightarrow L\stackrel{j}\rightarrow P\stackrel{p}\rightarrow N\rightarrow 0$ is a short exact sequence of $A$-modules with $P$ projective and $\omega_N=0$. Then for each $n\geq 1$, $N$ belongs to ${^{\perp_{[1, n]}}(\bar{A}\mbox{-}{\rm Proj})}$ if and only if  $L$ belongs to ${^{\perp_{[1, n]}}(A\mbox{-}{\rm Proj})}$.
\end{prop}

\begin{proof}
Assume first that $n=1$. Take a short exact sequence $0\rightarrow N'\stackrel{a}\rightarrow E\stackrel{b}\rightarrow N\rightarrow 0$ of $\bar{A}$-modules with $E$ a projective $\bar{A}$-module. Take a short exact sequence $0\rightarrow L'\stackrel{j'}\rightarrow P'\stackrel{p'}\rightarrow N'\rightarrow 0$ of $A$-modules with $P'$ projective. There is a morphism $c\colon P\rightarrow E$ satisfying $p=b\circ c$. Consequently, we have the following commutative diagram with exact rows and columns.
\begin{align}\label{cd:horse}
\xymatrix{
  & 0\ar[d] & 0\ar[d] & 0\ar[d]  \\
 0  \ar[r] & L'\ar[d]_-{j'}\ar[r]^-{a'} & Q\ar[d] \ar[r] & L \ar[d]^-{j}\ar[r] & 0\\
 0  \ar[r] & P'\ar[d]_-{p'}\ar[r]^-{\binom{1}{0}} & P'\oplus P\ar[d]^-{(a\circ p', c)} \ar[r]^-{(0, 1)} & P \ar[d]^-{p}\ar[r] & 0\\
 0\ar[r] & N' \ar[d]\ar[r]^-{a} & E\ar[d] \ar[r]^-{b} & N \ar[r]\ar[d] & 0 \\
 & 0 & 0 & 0
}\end{align}
Since $E$ has projective dimension one as an $A$-module, the $A$-module $Q$ is projective.

We apply Proposition~\ref{prop:approx} to the left half of the diagram, and obtain the following fact: $a'$ is a left $(A\mbox{-Proj})$-approximation if and only if $a$ is a left $(\bar{A}\mbox{-Proj})$-approximation. Here, we use the fact that $\binom{1}{0}$ is a split monomorphism. It follows from Lemma~\ref{lem:ext1} that  $N$ belongs to ${^{\perp_1}(\bar{A}\mbox{-}{\rm Proj})}$ if and only if  $L$ belongs to ${^{\perp_1}(A\mbox{-}{\rm Proj})}$.

For the general case, we observe from the bottom row in (\ref{cd:horse}) that $N$ belongs to  ${^{\perp_{[1, n+1]}}(\bar{A}\mbox{-}{\rm Proj})}$ if and only if $N'$ belongs to ${^{\perp_{[1, n]}}(\bar{A}\mbox{-}{\rm Proj})}$ and $a$ is a left $(\bar{A}\mbox{-Proj})$-approximation. Similarly, from the top row, we have that  the $A$-module $L$ belongs to  ${^{\perp_{[1, n+1]}}(A\mbox{-}{\rm Proj})}$ if and only if $L'$ belongs to ${^{\perp_{[1, n]}}(A\mbox{-}{\rm Proj})}$ and $a'$ is a left $(A\mbox{-Proj})$-approximation. Then we infer the general case  by induction on $n$.
\end{proof}

We have the following immediate consequence.

\begin{cor}\label{cor:ortho}
Assume that  $0\rightarrow L\stackrel{j}\rightarrow P\stackrel{p}\rightarrow N\rightarrow 0$ is a short exact sequence of $A$-modules with $P$ projective and $\omega_N=0$. Then  $N$ belongs to ${^{\perp}(\bar{A}\mbox{-}{\rm Proj})}$ if and only if  $L$ belongs to ${^{\perp}(A\mbox{-}{\rm Proj})}$.
\end{cor}

The  following main result of this section implies that for an $\bar{A}$-module $N$, it is Gorenstein projective if and only if its first syzygy  $\Omega_A(N)$ is a Gorenstein projective $A$-module.

\begin{thm}\label{thm:GP}
Assume that  $\eta\colon 0\rightarrow L\stackrel{j}\rightarrow P\stackrel{p}\rightarrow N\rightarrow 0$ is a short exact sequence of $A$-modules with $P$ projective and $\omega_N=0$. Then  $N$ belongs to $\bar{A}\mbox{-}{\rm GProj}$ if and only if  $L$ belongs to $A\mbox{-}{\rm GProj}$.
\end{thm}

\begin{proof}
For the ``only if" part, we assume that $N$ is a Gorenstein projective $\bar{A}$-module. By $N\in {^\perp(\bar{A}\mbox{-Proj})}$ and Corollary~\ref{cor:ortho}, we deduce $L\in {^\perp(A\mbox{-Proj})}$.

We have a short exact sequence $0\rightarrow N\stackrel{a} \rightarrow E\stackrel{b}\rightarrow N_1\rightarrow 0$ of $\bar{A}$-modules with $E$ projective and $N_1$ Gorenstein projective. Take a short exact sequence $0\rightarrow L_1\stackrel{j_1}\rightarrow P_1\stackrel{p_1}\rightarrow N_1\rightarrow 0$ of $A$-modules with $P_1$ projective. Therefore, we have a commutative diagram with exact rows and columns.
\begin{align*}
\xymatrix{
  & 0\ar[d] & 0\ar[d] & 0\ar[d]  \\
 0  \ar[r] & L\ar[d]_-{j}\ar[r] & Q\ar[d] \ar[r] & L_1 \ar[d]^-{j_1}\ar[r] & 0\\
 0  \ar[r] & P\ar[d]_-{p}\ar[r]^-{\binom{1}{0}} & P\oplus P_1\ar[d]^-{(a\circ p, c)} \ar[r]^-{(0, 1)} & P_1 \ar[d]^-{p_1}\ar[r] & 0\\
 0\ar[r] & N \ar[d]\ar[r]^-{a} & E\ar[d] \ar[r]^-{b} & N_1 \ar[r]\ar[d] & 0 \\
 & 0 & 0 & 0
}\end{align*}
Here, $c\colon P_1\rightarrow E$ is a morphism satisfying $p_1=b\circ c$. Consider the rightmost column in the diagram above.  By $N_1\in {^\perp(\bar{A}\mbox{-Proj})}$ and Corollary~\ref{cor:ortho}, we deduce $L_1\in {^\perp(A\mbox{-Proj})}$. By iterating the argument above, we obtain short exact sequences
$$0\longrightarrow L_i\longrightarrow P_i\longrightarrow L_{i+1}\longrightarrow 0$$
of $A$-modules with $P_i$ projective and $L_{i+1}\in  {^\perp(A\mbox{-Proj})}$, for all $i\geq 1$. By Lemma~\ref{lem:GP}, we deduce that the $A$-module $L$ is Gorenstein projective.

For the ``if" part, we assume that the $A$-module $L$ is Gorenstein projective. By $L\in {^\perp(A\mbox{-Proj})}$ and Corollary~\ref{cor:ortho}, we deduce $N\in {^\perp(\bar{A}\mbox{-Proj})}$.

 Take a short exact sequence $0\rightarrow L \stackrel{x} \rightarrow Q \rightarrow  L'_1\rightarrow 0$ with $Q$ projective and $L'_1$ Gorenstein projective. Consider the following diagram obtained by taking a pushout.
\[\xymatrix{
0\ar[r] & L \ar[d]_-{j} \ar[r]^-x & Q \ar@{.>}[d]\ar[r] & L'_1\ar@{=}[d]\ar[r] & 0\\
0\ar[r] & P \ar@{.>}[r] &L_1 \ar[r] & L'_1\ar[r] & 0
}\]
Since ${\rm Ext}_A^1(L'_1, P)=0$, we infer that $L_1\simeq P\oplus L'_1$. In particular, $L_1$ is Gorenstein projective. Applying Lemma~\ref{lem:factor} to the exact sequence $\eta$ and the  morphism $x\colon L\rightarrow Q$, we obtain a morphism $y\colon P\rightarrow {^\sigma(Q)}$ satisfying $y\circ j=\omega_Q\circ x$. Consequently, we obtain the following commutative diagram with exact rows and columns.
\[\xymatrix{
  & 0\ar[d] & 0\ar[d] &   \\
 0  \ar[r] & L\ar[d]_-{j}\ar[r]^{\binom{x}{j}} & Q\oplus P\ar[d]^-{{\tiny \begin{pmatrix} \omega_Q & 0\\
 0 & 1\end{pmatrix}}} \ar[r] & L_1 \ar[d]^-{j_1}\ar[r] & 0\\
 0  \ar[r] & P\ar[d]_-{p}\ar[r]^-{\binom{y}{1}} & {^\sigma(Q)}\oplus P\ar[d]^-{(q, 0)} \ar[r]^-{(1, -y)} & {^\sigma(Q)} \ar[d]\ar[r] & 0\\
& N \ar[d]\ar[r]^-{a} & E_1\ar[d] \ar[r] & N_1 \ar[r]\ar[d] & 0 \\
 & 0 & 0 & 0
}\]
Here, $q\colon {^\sigma(Q)}\rightarrow E_1$ is the cokernel of $\omega_Q\colon Q\rightarrow {^\sigma(Q)}$. It follows that $E_1$ is a projective $\bar{A}$-module. By the Five Lemma, we infer that the kernel of $a$ is isomorphic to the kernel of $j_1$. Since $\omega_N=0$ and $L_1$ is $\omega$-torsionfree, the common kernel has to be zero. In other words, both $a$ and $j_1$ are monomorphisms.

Consider the rightmost column in the diagram above. By $L_1\in {^\perp(A\mbox{-Proj})}$ and Corollary~\ref{cor:ortho}, we deduce $N_1\in {^\perp(\bar{A}\mbox{-Proj})}$. In other words, we obtain a short exact sequence of $\bar{A}$-modules
$$0\longrightarrow N \longrightarrow E \longrightarrow N_1\longrightarrow 0$$
with $E$ projective and $N_1\in {^\perp(\bar{A}\mbox{-Proj})}$.

By  iterating the argument above, we obtain short exact sequences
$$0\longrightarrow N_i\longrightarrow E_i\longrightarrow N_{i+1}\longrightarrow 0$$
of $\bar{A}$-modules with $E_i$ projective and $N_{i+1}\in  {^\perp(\bar{A}\mbox{-Proj})}$ for all $i\geq 1$. By Lemma~\ref{lem:GP} again, we infer  that the $\bar{A}$-module $N$ is Gorenstein projective.
\end{proof}

\subsection{Gorenstein projective dimensions}

For any $A$-module $L$, we denote by ${\rm Gpd}_A(L)$ its \emph{Gorenstein projective dimension}, which takes value in $\mathbb{N}\cup \{+\infty\}$. For each $n\geq 0$, we have ${\rm Gpd}_A(L)\leq n$ if and only if there exists an exact sequence $0\rightarrow G^{-n}\rightarrow \cdots \rightarrow G^{-1}\rightarrow G^0\rightarrow L\rightarrow 0$ of $A$-modules with each $G^i$ Gorenstein projective. In particular, ${\rm Gpd}_A(L)=0$ if and only if $L$ is Gorenstein projective.

 The ring $A$ is called \emph{left $d$-Gorenstein} for some $d\geq 0$ provided that ${\rm Gpd}_A(L)\leq d$ for any $A$-module $L$; see \cite[Definition~6.8 and Theorem~6.9]{Bel00}. The main examples of left $d$-Gorenstein rings are $d$-Gorenstein rings. Here, we recall that a two-sided noetherian ring $A$ is \emph{$d$-Gorenstein}, if $A$ has finite selfinjective dimension at most $d$ on both sides.

The following immediate consequence of \cite[Proposition~2.18]{Holm} will be useful.

\begin{lem}\label{lem:Gpd}
Assume that $0\rightarrow L\rightarrow M\rightarrow N\rightarrow 0$ is an exact sequence of $A$-modules. Suppose that ${\rm Gpd}_A(M)\leq i_0$ and that  $j_0\geq 0$. Then ${\rm Gpd}_A(L)\leq i_0+j_0$ if and only if ${\rm Gpd}_A(N)\leq i_0+j_0+1$. \hfill $\square$
\end{lem}

We now somehow reformulate Theorem~\ref{thm:GP} as a change-of-rings theorem on Gorenstein projective dimensions. It is a noncommutative analogue of \cite[Theorem~4.1]{BM}, and might be viewed as an infinite version of \cite[(2.2.8)~Theorem (Change of Rings)]{Christ}; compare \cite[(4.32)~Lemma]{ABr}.

\begin{thm}\label{thm:Gpd}
Let $N$ be a nonzero $\bar{A}$-module. Then we have ${\rm Gpd}_A(N)={\rm Gpd}_{\bar{A}}(N)+1$. Consequently, if $A$ is left $d$-Gorenstein for $d\geq 1$, then $\bar{A}$ is left $(d-1)$-Gorenstein.
\end{thm}

\begin{proof}
Since $\omega_N=0$, $N$ is not isomorphic to a submodule of any projective $A$-module. In particular, $N$ is a not Gorenstein projective $A$-modules, or equivalently, ${\rm Gpd}_A(N)\geq 1$.

For the required result, it suffices to prove the following claim: for any $\bar{A}$-module $N$ and each $n\geq 0$, we have that ${\rm Gpd}_{\bar{A}}(N)\leq n$ if and only if ${\rm Gpd}_{A}(N)\leq n+1$. We will prove the claim by induction on $n$. For this end, we fix a short exact sequence of $A$-modules
\begin{align}\label{equ:LPN}
0\longrightarrow L\stackrel{j}\longrightarrow P\stackrel{p}\longrightarrow N\longrightarrow 0
\end{align}
with  $P$ projective.

 Recall that ${\rm Gpd}_{\bar{A}}(N)\leq 0$ if and only if $N\in \bar{A}\mbox{-GProj}$. By Theorem~\ref{thm:GP}, the latter condition is equivalent to $L\in A\mbox{-GProj}$. Then we are done by the following result: applying   Lemma~\ref{lem:Gpd} to (\ref{equ:LPN}) with $i_0=0=j_0$, we have that $L\in A\mbox{-GProj}$ if and only if ${\rm Gpd}_A(N)\leq 1$. This proves the base case for the induction.

Assume that the claim for  the case $n_0$ is already established. We take a short exact sequence of $\bar{A}$-modules
\begin{align}\label{equ:N'N}
0\longrightarrow N'\stackrel{a}\longrightarrow E \stackrel{b}\longrightarrow N\longrightarrow 0
\end{align}
with  $E$ a projective $\bar{A}$-module. Then we obtain the commutative diagram (\ref{cd:horse}). We observe from (\ref{equ:N'N}) that ${\rm Gpd}_{\bar{A}}(N)\leq n_0+1$ is equivalent to ${\rm Gpd}_{\bar{A}}(N')\leq  n_0$. Applying the induction hypothesis to $N'$, the latter condition is equivalent to ${\rm Gpd}_A(N')\leq n_0+1$.  The $A$-module $E$ in (\ref{equ:N'N}) is of projective dimension one.  Therefore, applying Lemma~\ref{lem:Gpd} to (\ref{equ:N'N}) with $i_0=1$ and $j_0=n_0$, we have  that ${\rm Gpd}_A(N')\leq n_0+1$ if and only if ${\rm Gpd}_A(N)\leq n_0+2$. Combining all these results, we infer that ${\rm Gpd}_{\bar{A}}(N)\leq n_0+1$ if and only if ${\rm Gpd}_A(N)\leq n_0+2$. This completes the proof.
\end{proof}

We denote by ${\rm pd}_A(N)$ the projective dimension of an $A$-module $L$.

\begin{cor}\label{cor:pd1}
Let $N$ be a nonzero Gorenstein projective $\bar{A}$-module. Then ${\rm pd}_A(N)< +\infty$ if and only if ${\rm pd}_A(N)=1$.
\end{cor}

\begin{proof}
 The ``if" part is trivial. Assume that ${\rm pd}_A(N)< +\infty$. Then we have ${\rm Gpd}_A(N)={\rm pd}_A(N)$ by \cite[Proposition~2.27]{Holm}.  Since $N$ is a Gorenstein projective $\bar{A}$-module, Theorem~\ref{thm:Gpd} implies that ${\rm Gpd}_A(N)=1$. This implies the ``only if" part.
\end{proof}

The following result might be viewed as a partial  infinite version of  \cite[(1.4.6)~Corollary]{Christ}; compare \cite[(4.31)~Corollary]{ABr}.

\begin{cor}
Let $M$ be a $\omega$-torsionfree $A$-module. Then we have ${\rm Gpd}_{\bar{A}}(M/{\omega M})\leq {\rm Gpd}_A(M)$. In particular, if $M$ belongs to $A\mbox{-}{\rm GProj}$, then $M/{\omega M}$ belongs to $\bar{A}\mbox{-}{\rm GProj}.$
\end{cor}

\begin{proof}
Since $\omega$ is $M$-regular, the following sequence is short exact.
$$0\longrightarrow {^{\sigma^{-1}}(M)}\stackrel{\omega_{^{\sigma^{-1}}(M)}}\longrightarrow M\longrightarrow M/{\omega M}\longrightarrow 0$$
Here, we identify $M$ with ${^\sigma(^{\sigma^{-1}}(M))}$. We observe that $M$ and $^{\sigma^{-1}}(M)$ have the same Gorenstein projective dimension. Therefore, by Lemma~\ref{lem:Gpd} we have ${\rm Gpd}_A(M/{\omega M})\leq {\rm Gpd}_A(M)+1$. Then the required inequality follows from Theorem~\ref{thm:Gpd}. The final statement is immediate, since any Gorenstein projective $A$-module is $\omega$-torsionfree.
\end{proof}

The following result slightly  strengthens  Theorem~\ref{thm:GP}.

\begin{thm}\label{thm:GP-2}
Assume that  $0\rightarrow L \rightarrow G \rightarrow N\rightarrow 0$ is a short exact sequence of $A$-modules with $G\in A\mbox{-}{\rm GProj}$ and $\omega_N=0$. Then  $N$ belongs to $\bar{A}\mbox{-}{\rm GProj}$ if and only if  $L$ belongs to $A\mbox{-}{\rm GProj}$.
\end{thm}

\begin{proof}
If $L$ belongs to $A\mbox{-}{\rm GProj}$, we have ${\rm Gpd}_A(N)\leq 1$. Then Theorem~\ref{thm:Gpd} implies that ${\rm Gpd}_{\bar{A}}(N)\leq 0$, that is,  $N$ belongs to $\bar{A}\mbox{-}{\rm GProj}$.

Conversely, if $N$ belongs to $\bar{A}\mbox{-}{\rm GProj}$, Theorem~\ref{thm:Gpd} implies ${\rm Gpd}_A(N)\leq 1$. Then by \cite[(3.12)~Lemma]{ABr} or \cite[Theorem~2.20(iv)]{Holm}, we infer that $L$ belongs to $A\mbox{-}{\rm GProj}$.
\end{proof}

\section{The category of module factorizations}

In this section, we study the category of module factorizations of $\omega$. It turns out that it is equivalent to the module category over a matrix ring; see Proposition~\ref{prop:F-Gamma}. We prove that the subcategory formed by module factorizations with Gorenstein projective components is Frobenius exact; see Proposition~\ref{prop:GP}

\subsection{Module factorizations as modules}

Recall that a \emph{module factorization} of $\omega$ over $A$ is a
quadruple $X=(X^0, X^1; d_X^0 d_X^1)$, where each $X^i$ is a  left $A$-module, $d_X^0\colon X^0\rightarrow X^1$ and $d_X^1\colon X^1\rightarrow {^\sigma(X^0)}$  are morphisms of $A$-modules satisfying
$$d_X^1\circ d_X^0=\omega_{X^0} \mbox{ and }  {^\sigma(d_X^0)}\circ d_X^1=\omega_{X^1}.$$
The module factorization might be visualized as follows.
$$X^0\stackrel{d_X^0}\rightarrow X^1\stackrel{d_X^1}\rightsquigarrow X^0$$
We call $X^i$ the $i$-th component of $X$, and $d_X^i$ the $i$-th differntial.

Let $Y=(Y^0, Y^1; d_Y^0, d_Y^1)$ be another module factorization. A morphism $(f^0, f^1)\colon X\rightarrow Y$ between module factorizations consists of two morphisms $f^i\colon X^i\rightarrow Y^i$ of $A$-modules, which satisfy
$$f^1\circ d_X^0=d_Y^0\circ f^0 \mbox{  and } {^\sigma(f^0)}\circ d_X^1=d_Y^1\circ f^1.$$
The composition of morphisms is given componentwise. We denote by $\mathbf{F}(A; \omega)$ the category of module factorizations of $\omega$.

For each module factorization $X$, we define its \emph{shifted module factorization} $S(X)=(X^1, {^\sigma(X^0)}; -d_X^1, -{^\sigma(d_X^0)})$. This gives rise to the shift endofunctor
$$S\colon \mathbf{F}(A; \omega)\longrightarrow \mathbf{F}(A; \omega),$$
which sends any morphism $(f^0, f^1)\colon X\rightarrow Y$ to $(f^1, {^\sigma(f^0)})\colon S(X)\rightarrow S(Y)$. It is an autoequivalence.  We mention that the two minus signs in $S(X)$ are consistent with the ones appearing in the  suspended cochain complexes.

\begin{exm}\label{exm:theta}
{\rm Each $A$-module $M$ gives rise to two \emph{trivial module factorizations}:
 $$\theta^0(M)=(M, M; {\rm Id}_M, \omega_M) \mbox{ and } \theta^1(M)=({^{\sigma^{-1}}(M)}, M; \omega_{^{\sigma^{-1}}(M)}, {\rm Id}_M).$$
Here, we have $\omega_{^{\sigma^{-1}}(M)}\colon {^{\sigma^{-1}}(M)}\rightarrow {^\sigma(^{\sigma^{-1}}(M))}=M$ and ${\rm Id}_M\colon M\rightarrow M={^\sigma(^{\sigma^{-1}}(M))}$.
These two module factorizations are visualized as follows:
$$M\stackrel{{\rm Id}_M}\rightarrow M\stackrel{\omega_M}\rightsquigarrow M \mbox{ and } {^{\sigma^{-1}}(M)}\stackrel{\omega_{^{\sigma^{-1}}(M)}}\rightarrow M\stackrel{{\rm Id}_M}\rightsquigarrow {^{\sigma^{-1}}(M)}.$$
This gives rise to two obvious functors
$$\theta^i\colon A\mbox{-Mod}\longrightarrow \mathbf{F}(A; \omega), \; i=0,1.$$
 Moreover, we have the following natural isomorphism of functors.
$$S\theta^1\simeq \theta^0$$
}
\end{exm}

For $i=0$ or $1$, we have the projection functor
$${\rm pr}^i\colon \mathbf{F}(A; \omega)\rightarrow A\mbox{-Mod}$$
sending $X=(X^0, X^1; d_X^0, d_X^1)$ to $X^i$.

\begin{lem}\label{lem:2adj}
The following statements hold.
\begin{enumerate}
    \item We have the adjoint pairs $(\theta^0, {\rm pr}^0)$ and $({\rm pr}^0, S\theta^0).$
    \item We have the adjoint pairs $(\theta^1, {\rm pr}^1)$ and $({\rm pr}^1, \theta^0)$.
\end{enumerate}
\end{lem}

\begin{proof}
We only prove (1). For any $A$-module $M$ and any module factorization $X=(X^0, X^1; d_X^0, d_X^1)$, we have a natural isomorphism
$${\rm Hom}_A(M, X^0)\longrightarrow {\rm Hom}_{\mathbf{F}(A; \omega)}(\theta^0(M), X), \quad g\mapsto (g, d_X^0\circ  g).$$
This yields the first adjoint pair. For the second one, we have the following natural isomorphism
$${\rm Hom}_A(X^0, M)\longrightarrow {\rm Hom}_{\mathbf{F}(A; \omega)}(X, S\theta^0(M))$$
sending $g$ to $(g, -{^\sigma(g)}\circ d_X^1)$. Here, the minus sign appears due to the one in the definition of the shift endofunctor $S$ on $\mathbf{F}(A; \omega)$.
\end{proof}

Denote by $M_2(A)$ the full matrix ring formed by $2\times 2$ matrices with entries in $A$. Consider the following subring of $M_2(A)$.
$$\Gamma=\begin{pmatrix}
    A & A \\
    A\omega & A
\end{pmatrix}$$
We mention that such a matrix ring appears in the representation theory  \cite{KZ} of orders.

\begin{exm}\label{exm:Gamma}
Set $e_1=\begin{pmatrix}
    1 & 0 \\
    0 & 0
\end{pmatrix}$ and $e_0=\begin{pmatrix}
    0 & 0 \\
    0 & 1
\end{pmatrix}$. We identify $\Gamma e_0$ with $\begin{pmatrix}
A\\
A
\end{pmatrix}$, which consists of column vectors. It is naturally a $\Gamma$-$A$-bimodule. For any $A$-module $M$, we identify the induced $\Gamma$-module $\Gamma e_0\otimes_A M$ with $\begin{pmatrix}
M\\
M
\end{pmatrix}$. Similarly, we identify $\Gamma e_1$ with$\begin{pmatrix}
A\\
A\omega
\end{pmatrix}$. Recall the following well-known isomorphism of left $A$-modules.
$$A\omega \otimes_A M\longrightarrow {^{\sigma^{-1}}(M)}, \quad aw\otimes_A y\mapsto a(^{\sigma^{-1}}(y))={^{\sigma^{-1}}(\sigma^{-1}(a)y)}$$
Using this isomorphism, we identify the induced $\Gamma$-module $\Gamma e_1\otimes_A M$ with $\begin{pmatrix}
M\\
{^{\sigma^{-1}}(M)}
\end{pmatrix}$. The explicit $\Gamma$-action on $\begin{pmatrix}
M\\
{^{\sigma^{-1}}(M)}
\end{pmatrix}$ is given as follows.
$$\begin{pmatrix}a_{11} & a_{12}\\
a_{21}\omega & a_{22}\end{pmatrix} \begin{pmatrix}
    x\\
    {^{\sigma^{-1}}(y)}
\end{pmatrix}=\begin{pmatrix}
    a_{11}x+a_{12}\omega y\\
    a_{21}(^{\sigma^{-1}}(x))+a_{22}(^{\sigma^{-1}}(y))
\end{pmatrix}
$$
We mention that projective $\Gamma$-modules are precisely direct summands of $(\Gamma e_0\otimes_A P)\oplus (\Gamma e_1\otimes_A Q)$ for some projective $A$-modules $P$ and $Q$.
\end{exm}

To each module factorization $X=(X^0, X^1; d_X^0, d_X^1)$, we associate a left $\Gamma$-module $\Phi(X)$ as follows. As an abelian group, $\Phi(X)=\begin{pmatrix}X^1\\ X^0\end{pmatrix}$ whose typical element is a column vector $\begin{pmatrix}x^1\\ x^0\end{pmatrix}$ with each $x^i\in X^i$. The left $\Gamma$-action is given by
$$\begin{pmatrix}a_{11} & a_{12}\\
                 a_{21}\omega & a_{22}\end{pmatrix} \begin{pmatrix}x^1\\ x^0\end{pmatrix}= \begin{pmatrix}a_{11}x^1+a_{12}d_X^0(x^0)\\ a_{21}(\xi_{X^0}^{-1}\circ d_X^1)(x^1)+a_{22}x^0\end{pmatrix}. $$
Here, $\xi_{X^0}\colon X_0\rightarrow {^\sigma(X_0)}$ is an isomorphism of abelian groups, sending $x$ to $^\sigma(x)$. To verify that this $\Gamma$-action is associative, we use the following identities:
\begin{align*}
    &(d_X^0\circ \xi_{X^0}^{-1}\circ d_X^1)(x^1)=wx^1, \quad (\xi_{X^0}^{-1}\circ d_X^1)(a_{11}x^1)=\sigma(a_{11})  (\xi_{X^0}^{-1}\circ d_X^1)(x^1)\\
    & \mbox{ and }  (\xi_{X^0}^{-1}\circ d_X^1)(a_{12}d_X^0(x^0))=\sigma(a_{12})wx^0.
\end{align*}
Each morphism $(f^0, f^1)\colon X\rightarrow Y$ of module factorizations gives rise to a morphism $\Phi(f^0, f^1)=\begin{pmatrix}
    f^1\\ f^0
\end{pmatrix}\colon \Phi(X)\rightarrow \Phi(Y)$ of left $\Gamma$-modules. Consequently, we have a well-defined functor $\Phi\colon \mathbf{F}(A; \omega)\rightarrow \Gamma\mbox{-Mod}$.

The following result might be deduced from the characterization in \cite[(1.5)~Theorem]{Green} on modules over Morita context rings.

\begin{prop}\label{prop:F-Gamma}
   The functor $\Phi \colon \mathbf{F}(A; \omega)\rightarrow \Gamma\mbox{-}{\rm Mod}$ above is an equivalence of categories.
\end{prop}

\begin{proof}
Each $\Gamma$-module $M$ yields a module factorization $\Psi(M)=(e_0 M, e_1 M; d^0, d^1)$, where $d^0\colon e_0M\rightarrow e_1M$ sends $x$ to $\begin{pmatrix}
    0 & 1 \\
    0 & 0
\end{pmatrix}x$, and $d^1\colon e_1M \rightarrow {^\sigma(e_0M)}$ sends $y$ to $^\sigma(\begin{pmatrix}
    0 & 0 \\
    \omega & 0
\end{pmatrix}y)$. It is routine to verify that $\Psi$ is naturally a functor, which is a quasi-inverse of $\Phi$.
\end{proof}

\begin{rem}\label{rem:proj}
By Example~\ref{exm:Gamma}, we identify $\Phi(\theta^0(P))$ with  $\Gamma e_0\otimes_A P$ for any  projective $A$-module $P$. Similarly, we identify $\Phi(\theta^1(Q))$ with $\Gamma e_1\otimes_A Q$ for any projective $A$-module $Q$. It follows from the equivalence $\Phi$ above that projective objects in $\mathbf{F}(A; \omega)$ are precisely direct summands of $\theta^0(P)\oplus \theta^1(Q)$ for projective $A$-modules $P$ and $Q$.
\end{rem}

\begin{rem}\label{rem:Gamma}
There is  a ring automorphism $\overline{\sigma}$ on $\Gamma$ given by
$$\overline{\sigma}\begin{pmatrix} a_{11} & a_{12} \\
                    a_{21}\omega & a_{22} \end{pmatrix}= \begin{pmatrix}a_{22} & -a_{21}\\
                                                                      -\sigma(a_{12}) \omega & \sigma(a_{11}) \end{pmatrix}.$$
We observe the following square,
\[\xymatrix{
{\bf F}(A; \omega) \ar[d]_-{S} \ar[rr]^-{\Phi} && \Gamma\mbox{-Mod}\ar[d]^-{^{\overline{\sigma}}(-)}\\
{\bf F}(A; \omega) \ar[rr]^-{\Phi} && \Gamma\mbox{-Mod}
}\]
which is commutative up to a natural isomorphism. Here, ${^{\overline{\sigma}}(-)}$ denotes the twist endofunctor associated to $\overline{\sigma}$.
\end{rem}

\subsection{Frobenius functors}

We refer to \cite{Mor65, CGN99, CDM02} for Frobenius functors. Let $\mathcal{C}$ and $\mathcal{D}$ be two additive categories. Fix two autoequivalences $\alpha\colon \mathcal{C}\rightarrow \mathcal{C}$ and $\beta\colon \mathcal{D}\rightarrow \mathcal{D}$. Let $F\colon \mathcal{C}\rightarrow \mathcal{D}$ and $G\colon \mathcal{D}\rightarrow \mathcal{C}$ be two additive functors. The pair $(F, G)$  is called a \emph{Frobenius pair} of type $(\alpha, \beta)$ if both $(F, G)$ and $(G, \beta F\alpha)$ are adjoint pairs. An additive functor $F$ is called a \emph{Frobenius functor} if it fits into a Frobenius pair $(F, G)$; in this situation, the functor $G$ is also Frobenius.

Let $\mathcal{C}'$ be another additive category. Denote by $\mathcal{C}\times\mathcal{C}'$ the product category: its objects are pairs $(X, X')$ with $X\in \mathcal{C}$ and $X'\in \mathcal{C}'$, and the Hom groups are given by
$${\rm Hom}_{\mathcal{C}\times \mathcal{C}'}((X, X'), (Y, Y'))= {\rm Hom}_\mathcal{C}(X, Y)\oplus {\rm Hom}_{\mathcal{C}'}(X', Y').$$
Let $\alpha'\colon \mathcal{C}'\rightarrow \mathcal{C}'$ be an autoequivalence. Then $\alpha\times \alpha'\colon \mathcal{C}\times \mathcal{C}'\rightarrow \mathcal{C}\times \mathcal{C}'$ is naturally an autoequivalence.

We assume that $F'\colon \mathcal{C}'\rightarrow \mathcal{D}$ and $G'\colon \mathcal{D}\rightarrow \mathcal{C}'$ form a Frobenius pair $(F', G')$ of type $(\alpha', \beta)$. The functor $(G, G')\colon \mathcal{D}\rightarrow \mathcal{C}\times \mathcal{C}'$ sends any object $C$ to $(G(C), G'(C))$. We define a functor $F+F'\colon \mathcal{C}\times \mathcal{C}'\rightarrow \mathcal{D}$ sending an object $(X, X')$ to $F(X)\oplus F'(X')$, and a  morphism $(f, f')$ to $F(f)\oplus F'(f')$.

The following general fact is standard.

\begin{lem}\label{lem:general}
Keep the assumptions as above. Then $(F+F', (G, G'))$ is a Frobenius pair of type $(\alpha\times \alpha', \beta)$. \hfill $\square$
\end{lem}

We will  consider the following functor
$${\rm pr}=({\rm pr}^0, {\rm pr}^1)\colon \mathbf{F}(A; \omega)\longrightarrow A\mbox{-Mod}\times A\mbox{-Mod}.$$
It might be viewed as a forgetful functor.

\begin{prop}\label{prop:Frobenius}
We have a Frobenius pair $(\theta^0+\theta^1, {\rm pr})$ of type $({\rm  Id}, S)$, which consists of two faithful functors.
\end{prop}

\begin{proof}
Recall from Example~\ref{exm:theta} the natural isomorphism $S\theta^0\simeq \theta^1$. Then the required Frobenius pair follows by combining Lemmas~\ref{lem:general} and \ref{lem:2adj}. It is direct to see that both functors $\theta^0+\theta^1$ and ${\rm pr}$ are faithful.
\end{proof}

\begin{rem}\label{rem:Frobenius}
Consider the direct product $A\times A$ of rings, which is viewed as the diagonal subring of $\Gamma$. We identify $(A\times A)\mbox{-Mod}$ with $A\mbox{-Mod}\times A\mbox{-Mod}$. Proposition~\ref{prop:F-Gamma} allows us to identity $\Gamma\mbox{-Mod}$ with $\mathbf{F}(A; \omega)$. Then the functor ${\rm pr}\colon \mathbf{F}(A; \omega)\rightarrow A\mbox{-Mod}\times A\mbox{-Mod}$ corresponds to the restriction functor $\Gamma\mbox{-Mod}\rightarrow (A\times A)\mbox{-Mod}$ along the inclusion $A\times A\hookrightarrow \Gamma$. Therefore, the Frobenius pair in Proposition~\ref{prop:Frobenius} and Remark~\ref{rem:Gamma} imply that  $A\times A\hookrightarrow \Gamma$ is a (generalized) Frobenius extension in the sense of \cite[Example~2.9]{CDM02}.
\end{rem}

\begin{prop}\label{prop:d-Gor}
Let $d\geq 0$. Then $A$ is left $d$-Gorenstein if and only if so is $\Gamma$.
\end{prop}

\begin{proof}
By the Frobenius pair in Proposition~\ref{prop:Frobenius}, we infer from \cite[Corollary~3.4]{Chen-Ren} that the abelian categories $\mathbf{F}(A; \omega)$ and $A\mbox{-Mod}\times A\mbox{-Mod}$ have the same Gorenstein global dimension.  We recall that a ring $R$ is left $d$-Gorenstein if and only if the Gorenstein global dimension of its module category $R\mbox{-Mod}$ is at most $d$.  In view of Remark~\ref{rem:Frobenius}, we infer that $\Gamma$ is left $d$-Gorenstein  if and only if so is $A\times A$. Finally, we observe that the direct product $A\times A$  is left $d$-Gorenstein if and only if so is $A$.  Then the required result follows immediately.
\end{proof}

For any abelian category $\mathcal{A}$ with enough projective objects, we denote by $\mathcal{GP}(\mathcal{A})$ the full subcategory formed by Gorenstein projective objects. For example, we have $\mathcal{GP}(A\mbox{-Mod})=A\mbox{-GProj}$ for any ring $A$. By \cite[Proposition~3.8]{Bel}, $\mathcal{GP}(\mathcal{A})$ is always a Frobenius exact category, whose projective-injective objects are precisely all projective objects in $\mathcal{A}$.

Consider the full subcategory $\mathbf{GF}(A; \omega)$ of $\mathbf{F}(A; \omega)$ formed by those module factorizations $X$ with both $X^i\in A\mbox{-GProj}$. Since the full subcategory $\mathbf{GF}(A; \omega)$ is closed under extensions in $\mathbf{F}(A; \omega)$, it is naturally an exact category.

\begin{prop}\label{prop:GP}
We have $\mathcal{GP}(\mathbf{F}(A; \omega))=\mathbf{GF}(A; \omega)$, which is equivalent to $\Gamma\mbox{-}{\rm GProj}$ as an exact category.  In particular, the exact category $\mathbf{GF}(A; \omega)$ is Frobenius, whose projective-injective objects are precisely direct summands of $\theta^0(P)\oplus \theta^1(Q)$ for projective $A$-modules $P$ and $Q$.
\end{prop}

\begin{proof}
We apply \cite[Theorem~3.2]{Chen-Ren} to the faithful Frobenius functor ${\rm pr}\colon \mathbf{F}(A; \omega)\rightarrow A\mbox{-Mod}\times A\mbox{-Mod}$, and obtain that a module factorization $X$ is a Gorenstein projective object if and only if so is  ${\rm pr}(X)=(X^0, X^1)$ in $A\mbox{-Mod}\times A\mbox{-Mod}$. Then we have the required equality. The equivalence $\Phi$ in Proposition~\ref{prop:F-Gamma} implies that $\mathcal{GP}(\mathbf{F}(A; \omega))$ is equivalent to  $\mathcal{GP}(\Gamma\mbox{-Mod})=\Gamma\mbox{-}{\rm GProj}$.
\end{proof}

\section{The p-null-homotopoical morphisms and stable categories}

In this section, we will introduce  p-null-homotopical morphisms between module factorizations. In particular, the corresponding stable category $\underline{\mathbf{GF}}(A; \omega)$ is canonically triangulated. We compute explicitly the syzygy of any module factorization.

\begin{defn}\label{defn:pnh}
A morphism $(f^0, f^1)\colon X\rightarrow Y$ between two module factorizations is \emph{p-null-homotopical}, if there exists two morphisms $h^0\colon X^0\rightarrow {^{\sigma^{-1}}(Y^1)}$ and $h^1\colon X^1\rightarrow Y^0$ of $A$-modules such that both $h^i$ factor through projective $A$-modules  and satisfy $f^0=h^1\circ d_X^0+{^{\sigma^{-1}}(d_Y^1)} \circ h^0$ and $f^1=d_Y^0\circ h^1+{^\sigma(h^0)}\circ d_X^1$.
\end{defn}

The following result  justifies the terminology above.

\begin{lem}\label{lem:p-null}
A morphism $(f^0, f^1)\colon X\rightarrow Y$ is p-null-homotopical if and only if it factors through $\theta^0(P)\oplus \theta^1(Q)$ for some projective $A$-modules $P$ and $Q$.
\end{lem}

\begin{proof}
For  the ``only if" part, we assume that the homotopy maps $h^0$ and $h^1$ have the following factorizations in $A\mbox{-Mod}$:
$$h^0\colon X^0\stackrel{a}\longrightarrow {^{\sigma^{-1}}(Q)}\xrightarrow{^{\sigma^{-1}}(b)} {^{\sigma^{-1}}(Y^1)} \mbox{ and } h^1\colon X^1\stackrel{p}\longrightarrow P\stackrel{q}\longrightarrow  Y^0.$$
Here, both $P$ and $Q$ are projective $A$-modules. It is routine to verify that the given morphism $(f^0, f^1)$ equals the sum of the following two morphisms:
$$X \xrightarrow{(a,\; {^\sigma(a)}\circ d_X^1)} \theta^1(Q) \xrightarrow{( \sigma^{-1}(d_Y^1\circ b),\; b)} Y \mbox{ and } X \xrightarrow{(p\circ d_X^0, \; p)} \theta^0(P) \xrightarrow{(q, \; d_Y^0\circ q)} Y.$$
It follows that $(f^0, f^1)$ factors through $\theta^0(P)\oplus \theta^1(Q)$, as required. By reversing the argument above, one proves the  ``if" part.
\end{proof}

We denote by $\underline{\mathbf{F}}(A;\omega)$ the stable category modulo p-null-homotopical morphisms, and  by $\underline{\mathbf{GF}}(A; \omega)$ the corresponding stable category of $\mathbf{GF}(A; \omega)$. By Proposition~\ref{prop:GP} and Lemma~\ref{lem:p-null}, the stable category $\underline{\mathbf{GF}}(A; \omega)$ coincides with the usual one for the Frobenius exact category $\mathbf{GF}(A; \omega)$. In particular, $\underline{\mathbf{GF}}(A; \omega)$ is canonically triangulated. For the same reason, we have the following observation.

\begin{prop}
The equivalence $\Phi\colon \mathbf{F}(A;\omega)\simeq \Gamma\mbox{-}{\rm Mod}$ in Proposition~\ref{prop:F-Gamma} induces an equivalence $\underline{\mathbf{F}}(A;\omega)\simeq \Gamma\mbox{-\underline{\rm Mod}}$, which restricts a triangle equivalence $$\underline{\mathbf{GF}}(A;\omega)\stackrel{\sim}\longrightarrow \Gamma\mbox{-\underline{\rm GProj}}.$$
\end{prop}

\begin{proof}
 By Lemma~\ref{lem:p-null} and Remark~\ref{rem:proj}, we infer that any morphism $(f_0, f_1)\colon X\rightarrow Y$ is p-null-homotopical if and only if $\Phi(f_0, f_1)\colon \Phi(X)\rightarrow \Phi(Y)$ factors through projective $\Gamma$-modules.  Then the equivalence between the stable categories follows immediately. In view of Proposition~\ref{prop:GP}, we have the restricted equivalence, which  is a triangle functor by the general result \cite[I.2.8]{Hap}.
\end{proof}

The suspension functor of $\underline{\mathbf{GF}}(A;\omega)$ is given by a quasi-inverse of the syzygy functor.  Let us describe the syzygy of a module factorization $X=(X^0, X^1; d_X^0, d_X^1)$.

We have the following commutative diagram with exact rows and both $P^i$ projective.
\[\xymatrix{
0\ar[r] & \Omega_A(X^0)\ar@{.>}[d]_-{\partial^0} \ar[r]^-{j^0} &  P^0\ar@{.>}[d]_-{\delta^0} \ar[r]^{\pi^0} & X^0 \ar[d]^-{d_X^0} \ar[r] & 0\\
0\ar[r] & \Omega_A(X^1) \ar@{.>}[d]_-{\partial^1} \ar[r]^-{j^1} &  P^1 \ar@{.>}[d]_-{\delta^1} \ar[r]^{\pi^1} & X^1\ar[d]^-{d_X^1} \ar[r] & 0\\
0\ar[r] & {^{\sigma}\Omega_A(X^0)} \ar[r]^-{^\sigma(j^0)} &  {^\sigma (P^0)} \ar[r]^{^\sigma(\pi^0)} & ^\sigma(X^0) \ar[r] & 0
}\]
Since $d_X^1\circ d_X^0=\omega_{X^0}$, by a diagram-chasing we have
$${^\sigma(\pi^0)}\circ (\omega_{P^0}-\delta^1\circ \delta^0)=0.$$
There is a unique morphism $h^0\colon P^0\rightarrow {^\sigma \Omega_A(X^0)}$ such that
$$\omega_{P^0}-\delta^1\circ \delta^0={^\sigma(j^0)}\circ h^0.$$
Similarly, there is a unique morphism $h^1\colon P^1\rightarrow {^\sigma \Omega_A(X^1)}$ such that
$$\omega_{P^1}-{^\sigma(\delta^0)}\circ \delta^1={^\sigma(j^1)}\circ h^1.$$
In what follows, we write $\tau=\sigma^{-1}$ and denote the identity map by $1$. We have the following commutative diagram with exact rows.
\[\xymatrix{
0\ar[r] & \Omega_A(X^0)\oplus {^{\tau}(P^1)} \ar[dd]_-{\begin{pmatrix}
    j^0 & -{^\tau(\delta^1)}\\
    \partial^0 & {^\tau(h^1)}
\end{pmatrix}} \ar[rr]^-{\begin{pmatrix} j^0 & -{^{\tau}(\delta^1)}\\
0 & 1\end{pmatrix}} && P^0 \oplus {^{\tau}(P^1)} \ar[dd]^-{\begin{pmatrix}
    1 & 0\\
    0& \omega_{^\tau(P^1)}
\end{pmatrix}} \ar[rr]^-{(\pi^0,\;  {^{\tau}(d_X^1\circ \pi^1)})} && X^0 \ar[dd]^-{d_X^0}\ar[r] & 0\\
\\
0\ar[r] & P^0\oplus \Omega_A(X^1) \ar[dd]_-{\begin{pmatrix}
    h^0 & \partial^1\\
    -\delta^0 & j^1
\end{pmatrix}} \ar[rr]^-{\begin{pmatrix}
    1 & 0\\
    -\delta^0 &  j^1
\end{pmatrix}} && P^0\oplus P^1 \ar[dd]^-{\begin{pmatrix}
    \omega_{P^0} & 0\\
    0 & 1
\end{pmatrix}}\ar[rr]^-{(d_X^0\circ \pi^0, \; \pi^1)} && X^1\ar[r] \ar[dd]^-{d_X^1} & 0\\
\\
0\ar[r] & {^\sigma{\Omega_A(X^0)}}\oplus P^1 \ar[rr]^-{\begin{pmatrix}
    {^\sigma(j^0)} & -\delta^1\\
    0 & 1
\end{pmatrix}} && {^\sigma(P^0)}\oplus P^1 \ar[rr]^-{(^\sigma(\pi^0), \; d_X^1\circ \pi^1)} && {^\sigma(X^0)} \ar[r] & 0
}\]
The middle column represents the module factorization $\theta^0(P^0)\oplus \theta^1(P^1)$, which is projective in $\mathbf{F}(A; \omega)$. It follows that the syzygy of $X$ is isomorphic to the following module factorization
$$(\Omega_A(X^0)\oplus {^\tau(P^1)}, P^0\oplus \Omega_A(X^1); \begin{pmatrix}
    j^0 & -{^\tau(\delta^1)}\\
    \partial^0 & {^\tau(h^1)}
\end{pmatrix}, \begin{pmatrix}
    h^0 & \partial^1\\
    -\delta^0 & j^1
\end{pmatrix} ).$$

\begin{rem}\label{rem:syzygy-S}
If both components $X^0$ and $X^1$ of $X$ are projective, we might take $\Omega_A(X^0)=0=\Omega_A(X^1)$, $P^i=X^i$ and $\delta^i=d_X^i$. It follows that the  syzygy of $X$ is isomorphic to
$$S^{-1}(X)=(^\tau(X^1), X^0; -{^\tau(d_X^1)}, -d_X^0).$$
Here, we recall that $S$ denotes the shift endofunctor on $\mathbf{F}(A; \omega)$; see Section~3.
\end{rem}

 \section{The cokernal functors}

 In this section, we study the cokernel functors from $\mathbf{F}(A; \omega)$ to $\bar{A}\mbox{-Mod}$. We prove that the stable module category over $\bar{A}$ is equivalent to a certain full subcategory of $\underline{\mathbf{F}}(A; \omega)$; see Theorem~\ref{thm:cok0}. This equivalence restricts to the one in Theorem~\ref{thm:cok0G}.

 Let $X=(X^0, X^1; d_X^0, d_X^1)$ be a module factorization. Then ${\rm Cok}^0(X)$ is defined to the cokernel of $d_X^0\colon X^0\rightarrow X^1$, which is naturally an $\bar{A}$-module. For a morphism $(f^0, f^1)\colon X\rightarrow Y$ between module factorizations, there exists a unique morphism $\overline{f^1}$ fitting into  the following commutative diagram with exact rows.
 \[
 \xymatrix{
 X^0\ar[d]_-{f^0}\ar[r]^-{d_X^0} & X^1\ar[d]^-{f^1} \ar[r] & {\rm Cok}(X) \ar@{.>}[d]^-{\overline{f^1}} \ar[r] & 0\\
 Y^0\ar[r]^-{d_Y^0} & Y^1 \ar[r] & {\rm Cok}(Y) \ar[r] & 0
 }\]
 We set ${\rm Cok}^0(f^0, f^1)=\overline{f^1}$. This gives rise to the \emph{zeroth cokernel functor}
 $${\rm Cok}^0\colon \mathbf{F}(A; \omega)\longrightarrow \bar{A}\mbox{-Mod}.$$

\begin{rem}
We also have the first cokernel functor
$${\rm Cok}^1\colon \mathbf{F}(A; \omega)\longrightarrow \bar{A}\mbox{-Mod},$$
which is given such that ${\rm Cok}^1(X)$ is the cokernel of $d_X^1\colon X^1\rightarrow {^\sigma(X^0)}$. There is a natural isomorphism ${\rm Cok}^0 S\simeq {\rm Cok}^1$. Therefore, it suffices to study the zeroth cokernel functor.
\end{rem}

 Denote by $\mathbf{F}^{\rm mp}(A; \omega)$ the full subcategory of $\mathbf{F}(A; \omega)$ formed by those module factorizations $X$ with  $d_X^0$ mono and $X^1$ projective. Here, ``mp" stands for ``mono and projective".

 \begin{rem}\label{rem:tf}
  Let $X=(X^0, X^1; d_X^0, d_X^1)$ be a module factorization. We observe that both $d_X^i$ are monomorphisms if and only if both $X^i$ are $\omega$-torsionfree. In particular, $X$  belongs to $\mathbf{F}^{\rm mp}(A; \omega)$ if and only if  $X^0$ is $\omega$-torsionfree and $X^1$ is projective.
 \end{rem}

 \begin{prop}\label{prop:full-dense}
 The zeroth cokernel functor ${\rm Cok}^0\colon \mathbf{F}^{\rm mp}(A; \omega)\rightarrow \bar{A}\mbox{-}{\rm Mod}$ is full and dense.
 \end{prop}

\begin{proof}
For the density of ${\rm Cok}^0$, we take an arbitrary $\bar{A}$-module $N$. Take a  short exact sequence
$$\xi\colon 0\longrightarrow L\stackrel{j}\longrightarrow P\stackrel{p}\longrightarrow N\longrightarrow 0$$
of $A$-modules with $P$ projective. Applying Lemma~\ref{lem:factor} to $\xi$ and ${\rm Id}_L$, we obtain a morphism $k\colon P\rightarrow {^\sigma(L)}$ satisfying $k\circ j=\omega_L$. We have
$$(\omega_P-{^\sigma(j)}\circ k)\circ j=\omega_P\circ j-{^\sigma(j)}\circ \omega_L=0.$$
Therefore, the map $\omega_P-{^\sigma(j)}\circ k\colon P\rightarrow {^\sigma(P)}$ factors through $N$. However, we have that $\omega_N=0$ and $(^\sigma(P))$ is $\omega$-torsionfree. By Lemma~\ref{lem:elem}, the factorization has to be trivial. In other words, we have $\omega_P-{^\sigma(j)}\circ k=0$. So, we obtain a module factorization $(L, P; j, k)$ in $\mathbf{F}^{\rm mp}(A; \omega)$.  In view of $\xi$, this is the required module factorization.

Take two module factorizations $X$ and $Y$. Fix an arbitrary morphism $g\colon {\rm Cok}^0(X)\rightarrow {\rm Cok}^0(Y)$ in $\bar{A}\mbox{-Mod}$.  The projectivitity of $X^1$ yields the following commutative diagram.
\[
\xymatrix{
0\ar[r] & X^0 \ar@{.>}[d]_-{f^0} \ar[r]^-{d_X^0} & X^1\ar@{.>}[d]^-{f^1}\ar[r] & {\rm Cok}(X)\ar[d]^-{g} \ar[r] & 0\\
0\ar[r] & Y^0 \ar[r]^-{d_Y^0} & Y^1\ar[r] & {\rm Cok}(Y) \ar[r] & 0
}\]
To show that $(f^0, f^1)\colon X\rightarrow Y$ is a morphism, it remains to prove $d_Y^1\circ f^1={^\sigma(f^0)}\circ d_X^1$. For this end, we have the following identity.
\begin{align*}
&{^\sigma(d_Y^0)}\circ (d_Y^1\circ f^1-{^\sigma(f^0)}\circ d_X^1)\\
&=\omega_{Y^1}\circ f^1-{^\sigma(d_Y^0\circ f^0)} \circ d_X^1 \\
&= \omega_{Y^1}\circ f^1-{^\sigma(f^1\circ d_X^0)} \circ d_X^1\\
&=\omega_{Y^1}\circ f^1-{^\sigma(f^1)}\circ \omega_{X^1}=0
\end{align*}
Since $d_Y^0$ is mono, we obtain the desired equality. This completes the proof of the fullness.
\end{proof}

\begin{lem}\label{lem:Cok}
Let $(f^0, f^1)\colon X\rightarrow Y$ be a morphism in $\mathbf{F}^{\rm mp}(A; \omega)$. Then the following two statements hold.
\begin{enumerate}
\item ${\rm Cok}^0(f^0, f^1)=0$ if and only if $(f^0, f^1)$ factors through $\theta^0(P)$ for some projective $A$-module $P$
\item The morphism ${\rm Cok}^0(f^0, f^1)$ factors though a projective $\bar{A}$-module if and only if $(f^0, f^1)$ is p-null-homotopical.
\end{enumerate}
\end{lem}

\begin{proof}
(1) The ``if" part is trivial, since ${\rm Cok}^0(\theta^0(P))=0$. For the ``only if" part, we assume that ${\rm Cok}^0(f^0, f^1)=0$. Then there is a morphism  $h\colon X^1\rightarrow Y^0$ satisfying $f^0=h\circ d_X^0$ and $f^1=d_Y^0\circ h$. Therefore, $(f^0, f^1)$ has the following factorization.
$$X \xrightarrow{(d_X^0, {\rm Id}_X)} \theta^0(X^1)\xrightarrow{(h, f^1)} Y$$
Here, we implicitly use the following identity
$$d_Y^1\circ f^1=d_Y^1\circ (d_Y^0\circ h)=\omega_{Y^0}\circ h={^\sigma(h)}\circ \omega_{X^1}.$$
Since $X^1$ is projective, we are done.

(2) Recall from Lemma~\ref{lem:p-null} that $(f^0, f^1)$ is p-null-homotopical if and only if it factors through $\theta^0(P)\oplus \theta^1(Q)$ for some projective $A$-modules $P$ and $Q$. Since ${\rm Cok}^0(\theta^0(P)\oplus \theta^1(Q))\simeq Q/{\omega Q}$ is a projective $\bar{A}$-module, the ``if" part follows immediately.

For the ``only if" part, we assume that  ${\rm Cok}^0(f^0, f^1)=a$ factors through a free $\bar{A}$-module $V$. We may assume that there is a free $A$-module $F$ with $F/{\omega F}=V$. So, we identify $V$ with ${\rm Cok}^0(\theta^1(F))$. Suppose that $x\colon {\rm Cok}^0(X)\rightarrow {\rm Cok}^0(\theta^1(F))$ and $y\colon {\rm Cok}^0(\theta^1(F))\rightarrow Y$ satisfy $a=y\circ x$. By the fullness in Proposition~\ref{prop:full-dense}, there are morphisms $(g^0, g^1)\colon X\rightarrow \theta^1(F)$ and $(k^0, k^1)\colon \theta^1(F)\rightarrow Y$ in $\mathbf{F}^{\rm mp}(A; \omega)$ such that $x={\rm Cok}^0(g^0, g^1)$ and $y={\rm Cok}^0(k^0, k^1)$. We conclude that the following morphism
$$(f^0, f^1)-(k^0, k^1)\circ (h^0, h^1)$$
is annihilated by ${\rm Cok}^0$. By (1), this morphism factors through $\theta^0(P)$ for some projective $A$-module $P$. Consequently, $(f^0, f^1)$ factors through $\theta^0(P)\oplus \theta^1(F)$. It is p-null-homotopical by Lemma~\ref{lem:p-null}.
\end{proof}

\begin{thm}\label{thm:cok0}
The zeroth cokernel functor ${\rm Cok}^0\colon \mathbf{F}^{\rm mp}(A; \omega)\rightarrow \bar{A}\mbox{-}{\rm Mod}$ induces an equivalence
$${\rm Cok}^0\colon \underline{\mathbf{F}}^{\rm mp}(A; \omega)\stackrel{\sim}\longrightarrow \bar{A}\mbox{-}\underline{\rm Mod}.$$
\end{thm}

\begin{proof}
By Lemmas~\ref{lem:p-null} and~\ref{lem:Cok}(2), the induced functor ${\rm Cok}^0\colon \underline{\mathbf{F}}^{\rm mp}(A; \omega)\rightarrow \bar{A}\mbox{-}\underline{\rm Mod}$ is well-defined and faithful. Its fullness and denseness follow from the ones in Proposition~\ref{prop:full-dense}. Then we are done.
\end{proof}

Recall that $\mathbf{G}^0\mathbf{F}(A; \omega)$ denotes the full subcategory of $\mathbf{G}\mathbf{F}(A; \omega)$ formed by those module factorizations $X=(X^0, X^1; d_X^0, d_X^1)$ with $X^1$ projective and $X^0$ Gorenstein projective. It is closed under extensions and contains all projective-injective objects in $\mathbf{G}\mathbf{F}(A; \omega)$. Moreover, it is closed under taking syzygies and (relative) cosyzgyies; see Section~4. It follows that $\mathbf{G}^0\mathbf{F}(A; \omega)$ is a Frobenius exact category. The stable category $\underline{\mathbf{G}^0\mathbf{F}}(A; \omega)$ is a triangulated subcategory of $\underline{\mathbf{GF}}(A; \omega)$.

\begin{thm}\label{thm:cok0G}
The zeroth cokernal functor ${\rm Cok}^0\colon \mathbf{F}^{\rm mp}(A; \omega)\rightarrow \bar{A}\mbox{-}{\rm Mod}$ induces a triangle equivalence
$${\rm Cok}^0\colon \underline{\mathbf{G}^0\mathbf{F}}(A; \omega)\stackrel{\sim}\longrightarrow \bar{A}\mbox{-}\underline{\rm GProj}.$$
\end{thm}

\begin{proof}
We observe that any Gorenstein projective $A$-module is $\omega$-torsionfree. In view of Remark~\ref{rem:tf}, we have $\mathbf{G}^0\mathbf{F}(A; \omega)\subseteq \mathbf{F}^{\rm mp}(A; \omega)$. By the equivalence in Theorem~\ref{thm:cok0}, the required one follows immediately from the following fact: for any $X$ in $\mathbf{F}^{\rm mp}(A; \omega)$, the $\bar{A}$-module ${\rm Cok}^0(X)$ is Gorenstein projective if and only if $X^0$ is a Gorenstein projective $A$-module. The latter fact is just a reformulation of Theorem~\ref{thm:GP}.

To see that the induced equivalence is a triangle functor, we observe that
$${\rm Cok}^0\colon \mathbf{G}^0\mathbf{F}(A; \omega)\longrightarrow \bar{A}\mbox{-}{\rm GProj}$$
is exact and sends projective objects to projective objects. Then we just apply the general result in \cite[I.2.8]{Hap}.
\end{proof}

\section{A recollement}

In this section, we obtain a recollement involving $\underline{\mathbf{GF}}(A; \omega)$ in Proposition~\ref{prop:recGF}. We prove that the zeroth cokernal functor induces an equivalence between $\bar{A}\mbox{-}\underline{\rm GProj}$ and a certain Verdier quotient of $\underline{\mathbf{GF}}(A; \omega)$; see Theorem~\ref{thm:quotient}.

Recall that a diagram of triangle functors between triangulated
categories

\[\xymatrix{
  \mathcal{T}'\;\ar[rr]|-{i}&&\;\mathcal{T}\; \ar[rr]|-{j}
  \ar@/^1.5pc/[ll]|{i_\rho}\ar@/_1.5pc/[ll]|{i_\lambda}&&
  \;\mathcal{T}''\ar@/^1.5pc/[ll]|{j_\rho}\ar@/_1.5pc/[ll]|{j_\lambda}
}\]
 forms a \emph{recollement} in the sense of \cite[1.4]{BBD}, provided that the
following conditions are satisfied:
\begin{enumerate}
\item[{\rm (R1)}] the pairs $(i_\lambda, i)$, $(i, i_\rho)$, $(j_\lambda, j)$  and $(j,j_\rho)$ are adjoint;
\item[{\rm (R2)}] the functors $i$, $j_\lambda$ and $j_\rho$ are fully faithful;
\item[{\rm (R3)}] ${\rm Im}\; i={\rm Ker}\;  j$.
\end{enumerate}
Here, for an additive functor $F$, ${\rm Im}\; F$ and ${\rm Ker}\; F$ denotes the essential
image and kernel of $F$, respectively.

The following fact is well known.

\begin{lem}\label{lem:rec}
Let $j\colon \mathcal{T}\rightarrow \mathcal{T}''$ be a triangle functor between triangulated category. Assume that $j$ has a left adjoint $j_\lambda$ and a right adjoint $j_\rho$, and that $j_\rho$ is fully faithful. Then we have a recollement.
\[\xymatrix{
  {\rm Ker}\; j\;\ar[rr]|-{\rm inc}&&\;\mathcal{T}\; \ar[rr]|-{j}
  \ar@/^1.5pc/[ll]|{{\rm inc}_\rho}\ar@/_1.5pc/[ll]|{{\rm inc}_\lambda}&&
  \;\mathcal{T}''\ar@/^1.5pc/[ll]|{j_\rho}\ar@/_1.5pc/[ll]|{j_\lambda}
}\]
Here, ``${\rm inc}$" denotes the inclusion functor. Moreover, ${\rm inc}_\rho$ induces a triangle equivalence
$$\mathcal{T}/{{\rm Im}\; j_\rho}\simeq {\rm Ker}\; j.$$
\end{lem}

\begin{proof}
The proof is very similar to the one of \cite[Lemma~2.4]{Chen-VB}, and relies on  \cite[Propositions~1.5 and 1.6]{BK} and \cite[Lemma~8.3]{Kel}.  By applying \cite[I.1.3]{GZ} twice, we observe  that the fully-faithfulness of $j_\rho$ implies that $j$ is a quotient functor, which in turn implies the fully-faithfulness of $j_\lambda$.
\end{proof}

\begin{rem}\label{rem:const}
The functor ${\rm inc}_\rho$ is given by the cocone of the unit of the adjoint pair $(j, j_\rho)$.  In other words, any object $X\in \mathcal{T}$, ${\rm inc}_\rho (X)$ is determined by the following exact triangle
 $${\rm inc}_\rho (X)\longrightarrow X\longrightarrow j_\rho j(X)\longrightarrow \Sigma {\rm inc}_\rho(X).$$
 Here, $\Sigma$ denotes the suspension functor of $\mathcal{T}$.
\end{rem}

The functors $\theta^i\colon A\mbox{-}{\rm GProj}\rightarrow \mathbf{GF}(A; \omega)$ are both exact and send projective objects to projective objects. Therefore, we have the induced triangle functors between the stable categories, still denoted by the same notation $\theta^i$. Similarly, we have the projection  functor ${\rm pr}^1\colon \underline{\mathbf{GF}}(A; \omega)\rightarrow A\mbox{-}\underline{\rm GProj}$, which is also a triangle functor.

\begin{prop}\label{prop:recGF}
We have the following recollement.
\[
\xymatrix{
\underline{\mathbf{G}^0\mathbf{F}}(A; \omega)  \ar[rr]|{\rm inc} &&  \ar@/_1.5pc/[ll]|{{\rm inc}_\lambda}  \ar@/^1.5pc/[ll]|{{\rm inc}_\rho}  \underline{\mathbf{G}\mathbf{F}}(A; \omega)  \ar[rr]|{\rm pr^1} && A\mbox{-}\underline{\rm GProj} \ar@/_1.5pc/[ll]|{\theta^1}  \ar@/^1.5pc/[ll]|{\theta^0}
} \]
In particular, ${\rm inc}_\rho$ induces a triangle equivalence
$$\underline{\mathbf{GF}}(A; \omega)/{{\rm Im}\; \theta^0}\stackrel{\sim}\longrightarrow \underline{\mathbf{G}^0\mathbf{F}}(A; \omega).$$
\end{prop}

\begin{proof}
Recall the adjoint pairs in Lemma~\ref{lem:2adj}. By \cite[Lemma~2.3]{Chen-VB},  the induced functors $\theta^1, {\rm pr}^1$ and $\theta^0$ form two adjoint pairs; moreover, both induced $\theta^i$ are fully faithful. The essential kernel of ${\rm pr}^1 \colon \underline{\mathbf{GF}}(A; \omega)\rightarrow A\mbox{-}\underline{\rm GProj}$ clearly equals  $\underline{\mathbf{G}^0\mathbf{F}}(A; \omega)$. Then the required recollement follows from Lemma~\ref{lem:rec}.
\end{proof}

Let $X=(X^0, X^1; d_X^0, d_X^1)$ be a  module factorization  in $\mathbf{GF}(A; \omega)$. Since both $X^i$ are Gorenstein projective and $d_X^0$ is a monomorphism, then Theorem~\ref{thm:GP-2} implies that ${\rm Cok}^0(X)$ is a Gorenstein projective $\bar{A}$-module. Hence the zeroth cokernel functor ${\rm Cok}^0\colon \mathbf{F}(A; \omega)\rightarrow \bar{A}\mbox{-Mod}$ restricts to the following one.
$$\widetilde{\rm Cok}^0\colon \mathbf{GF}(A; \omega)\longrightarrow \bar{A}\mbox{-}{\rm GProj}$$
It is exact and sends projective objects to projective objects. Then it induces the following triangle functor
$$\widetilde{\rm Cok}^0\colon \underline{\mathbf{GF}}(A; \omega)\longrightarrow \bar{A}\mbox{-}\underline{\rm GProj}$$
between the stable categories. Here, we use this notation $\widetilde{\rm Cok}^0$ to emphasize its difference from the equivalence  ${\rm Cok}^0\colon\underline{\mathbf{G}^0\mathbf{F}}(A; \omega)\rightarrow \bar{A}\mbox{-}\underline{\rm GProj}$ in Theorem~\ref{thm:cok0G}. Indeed, the restriction of $\widetilde{\rm Cok}^0$ on $\underline{\mathbf{G}^0\mathbf{F}}(A; \omega)$ yields the equivalence ${\rm Cok}^0$.

The key observation is as follows.

\begin{lem}\label{lem:compo}
The composition of ${\rm Cok}^0\colon\underline{\mathbf{G}^0\mathbf{F}}(A; \omega)\rightarrow \bar{A}\mbox{-}\underline{\rm GProj}$ in Theorem~\ref{thm:cok0G} with ${\rm inc}_\rho\colon \underline{\mathbf{GF}}(A; \omega)\rightarrow \underline{\mathbf{G}^0\mathbf{F}}(A; \omega)$ in Proposition~\ref{prop:recGF} is isomorphic to the above functor $\widetilde{\rm Cok}^0\colon \underline{\mathbf{GF}}(A; \omega)\rightarrow \bar{A}\mbox{-}\underline{\rm GProj}$.
\end{lem}

\begin{proof}
Recall from Remark~\ref{rem:const} the construction of ${\rm inc}_\rho$. It follows that for any $X\in \mathbf{GF}(A; \omega)$, there is a functorial exact triangle in $\underline{\mathbf{GF}}(A; \omega)$.
$${\rm inc}_\rho(X)\longrightarrow X \longrightarrow \theta^0(X^1)\longrightarrow \Sigma{\rm inc}_\rho(X)$$
Applying the triangle functor $\widetilde{\rm Cok}^0$ to this triangle, we have an exact triangle in $\bar{A}\mbox{-}\underline{\rm GProj}$.
$${\rm Cok}^0 {\rm inc}_\rho(X)\longrightarrow \widetilde{\rm Cok}^0(X) \longrightarrow \widetilde{\rm Cok}^0 \theta^0(X^1)\longrightarrow \Sigma {\rm Cok}^0{\rm inc}_\rho(X)$$
Since $\widetilde{\rm Cok}^0 \theta^0(X^1)$ is isomorphic to zero, we infer from \cite[I.1.7]{Hap} that the natural morphism
$${\rm Cok}^0 {\rm inc}_\rho(X)\longrightarrow \widetilde{\rm Cok}^0(X)$$
is the required isomorphism
\end{proof}

Denote by $\mathcal{N}$ the full subcategory of $\underline{\mathbf{GF}}(A; \omega)$ formed by those objects that are isomorphic to $\theta^0(G)$ for some Gorenstein projective $A$-modules $G$. It follows that $\mathcal{N}$ might be identified with the essential image of $\theta^0\colon A\mbox{-}\underline{\rm GProj}\rightarrow \underline{\mathbf{GF}}(A; \omega)$.

\begin{thm}\label{thm:quotient}
The above functor  $\widetilde{\rm Cok}^0\colon \underline{\mathbf{GF}}(A; \omega)\rightarrow \bar{A}\mbox{-}\underline{\rm GProj}$ induces a triangle equivalence
$$ \underline{\mathbf{GF}}(A; \omega)/\mathcal{N}\stackrel{\sim}\longrightarrow \bar{A}\mbox{-}\underline{\rm GProj}.$$
\end{thm}

\begin{proof}
We identify $\mathcal{N}$ with ${\rm Im}\; \theta^0$. Proposition~\ref{prop:recGF} implies that ${\rm inc}_\rho$ induces a triangle equivalence
$$ \underline{\mathbf{GF}}(A; \omega)/{\mathcal{N}}\stackrel{\sim}\longrightarrow \underline{\mathbf{G}^0\mathbf{F}}(A; \omega).$$
Composing this equivalence with the equivalence ${\rm Cok}^0\colon\underline{\mathbf{G}^0\mathbf{F}}(A; \omega)\rightarrow \bar{A}\mbox{-}\underline{\rm GProj}$ in Theorem~\ref{thm:cok0G}, we obtain a triangle equivalence
$$ \underline{\mathbf{GF}}(A; \omega)/{\mathcal{N}}\stackrel{\sim}\longrightarrow \bar{A}\mbox{-}\underline{\rm GProj}.$$
In view of Lemma~\ref{lem:compo}, we infer that this composite triangle equivalence is actually induced by $\widetilde{\rm Cok}^0$.
\end{proof}

\section{The finite case: matrix factorizations}

In this section, we restrict the equivalence in Theorem~\ref{thm:cok0G} to matrix factorizations, and strengthen Theorem~C in Introduction; see Theorem~\ref{thm:finite}.

Recall from \cite[Definition~2.2]{CCKM} that a module factorization $X=(X^0, X^1; d_X^0, d_X^1)$ is called a \emph{matrix factorization} if each $X^i$ is a finitely generated projective $A$-module.  Denote by $\mathbf{MF}(A; \omega)$ the full subcategory formed by matrix factorizations, and by $\underline{\mathbf{MF}}(A; \omega)$ its stable category, which is a triangulated subcategory of $\underline{\mathbf{GF}}(A; \omega)$. We emphasize that the suspension functor of $\underline{\mathbf{MF}}(A; \omega)$ is restricted from the shift endofunctor $S$; compare \cite[Proposition~5.7]{CCKM} and Remark~\ref{rem:syzygy-S}. We mention matrix factorizations of locally free sheaves over schemes in \cite{PV,BW}.

Let $A$ be any ring. A totally acyclic complex $P^\bullet$ of $A$-modules is called \emph{locally finite} if each component $P^i$ is finitely generated. Denote by $A\mbox{-Gproj}$ the full subcategory of $A\mbox{-mod}$ formed by modules of the form $Z^0(P^\bullet)$ for some locally finite totally acyclic complex $P^\bullet$. The modules in $A\mbox{-Gproj}$ are sometimes called \emph{totally reflexive modules}.  The category $A\mbox{-Gproj}$ is a Frobenius exact category, whose projective-injective objects are precisely finitely generated projective $A$-module. The stable category $A\mbox{-}\underline{\rm Gproj}$ is a triangulated subcategory of $A\mbox{-}\underline{\rm GProj}$. We have $A\mbox{-Gproj}\subseteq A\mbox{-GProj}\cap A\mbox{-mod}$ and that the equality holds if $A$ is left coherent; see \cite[Lemma~3.4]{Chen11}.

Recall that $\bar{A}=A/{(\omega)}$. Denote by $\bar{A}\mbox{-Gproj}^{<+\infty}$ the full subcategory of $\bar{A}\mbox{-Gproj}$  consisting of totally reflexive $\bar{A}$-modules $N$ with ${\rm pd}_A(N)<+\infty$. It is closed under extensions in $\bar{A}\mbox{-Gproj}$ and becomes a Frobenius exact category. Consequently, the stable category $\bar{A}\mbox{-}\underline{\rm Gproj}^{<+\infty}$ is canonically triangulated.

For any $A$-module, we write $\overline{M}=M/{\omega M}$. The following result is analogous to \cite[Proposition~5.1]{Eis}.

\begin{lem}\label{lem:finite}
Let $X=(X^0, X^1; d_X^0, d_X^1)$ be a matrix factorization. Then the following complex of $\bar{A}$-modules
$$ \cdots \longrightarrow \overline{^{\sigma^{-1}}(X^1)}\longrightarrow \overline{X^0}\longrightarrow \overline{X^1} \longrightarrow \overline{^\sigma(X^0)}\longrightarrow \overline{^\sigma(X^1)}\longrightarrow  \cdots  $$
is locally finite totally acyclic, where the differentials are induced by $d_X^i$. Consequently, we have that ${\rm Cok}^0(X)$ belongs to $\bar{A}\mbox{-}{\rm Gproj}^{<+\infty}$. \hfill $\square$
\end{lem}

\begin{proof}
The result is due to \cite[Lemma~5.9]{MU}; compare \cite[Proposition~2.4]{CCKM}. For the final statement, we just observe  that ${\rm Cok}^0(X)$ is isomorphic to the cokernel of the induced morphism $\overline{X^0}\rightarrow \overline{X^1}$, which is further isomorphic to the kernel of $\overline{^\sigma(X^0)}\rightarrow \overline{^\sigma(X^1)}$. In other words, the  complex above is a complete resolution of ${\rm Cok}^0(X)$ up to suspension. Moreover, since $d_X^0$ is mono, we have a short exact sequence $0\rightarrow X^0\rightarrow X^1\rightarrow {\rm Cok}^0(X)\rightarrow 0 $ of $A$-modules with both $X^i$ projective.
\end{proof}

\begin{thm}\label{thm:finite}
The equivalence  ${\rm Cok}^0\colon\underline{\mathbf{G}^0\mathbf{F}}(A; \omega)\rightarrow \bar{A}\mbox{-}\underline{\rm GProj}$ in Theorem~\ref{thm:cok0G} restricts to  a triangle equivalence
$${\rm Cok}^0\colon \underline{\mathbf{MF}}(A; \omega)\longrightarrow \bar{A}\mbox{-}\underline{\rm Gproj}^{<+\infty}.$$

\end{thm}

\begin{proof}
Lemma~\ref{lem:finite} implies that the restricted cokernel functor ${\rm Cok}^0$ is well defined. By Theorem~\ref{thm:cok0G}, it is clearly fully faithful.

To see that  ${\rm Cok}^0$ is dense, we take any module $N\in \bar{A}\mbox{-}{\rm Gproj}^{<+\infty}$. The quotient ring $\bar{A}$ is a finitely presented $A$-module. Since $N$ is a finitely presented $\bar{A}$-module, we infer that $N$ is a finitely presented $A$-module. By Corollary~\ref{cor:pd1}, the $A$-module $N$ has projective dimension one. It follows that there is a short exact sequence of $A$-modules
\begin{align}\label{ses:N}
0\longrightarrow P^0\stackrel{d_P^0}\longrightarrow P^1\longrightarrow N\longrightarrow 0
\end{align}
with both $P^i$ finitely generated projective $A$-modules. By the same argument in the first paragraph in the proof of Proposition~\ref{prop:full-dense}, there is a morphism $d_P^1\colon P^1\rightarrow {^\sigma(P^0)}$ such that $P=(P^0, P^1; d_P^0, d_P^1)$ is a matrix factorization. By (\ref{ses:N}), we have an isomorphism  ${\rm Cok}^0(P)\simeq N$, as required.
\end{proof}

\begin{rem}
(1) When $A$ is commutative noetherian, Theorem~\ref{thm:finite} is due to \cite[Example~4.4]{Chen12} with a different proof.

(2) Assume that $A$ has finite global dimension. We have $ \bar{A}\mbox{-}\underline{\rm Gproj}= \bar{A}\mbox{-}\underline{\rm Gproj}^{<+\infty}$. Therefore,  Theorem~\ref{thm:finite}  strengthens Theorem~C in Introduction.

(3) Assume that $A$ is left coherent and that  any finitely presented Gorenstein projective $A$-module is projective. The same argument as in the proof of Theorem~\ref{thm:finite} yields $\bar{A}\mbox{-}\underline{\rm Gproj}=\bar{A}\mbox{-}\underline{\rm Gproj}^{<+\infty}$. Then we obtain a noncommutative version of \cite[Theorem~A]{Tak}.
\end{rem}

\section{An example: integral representations of finite groups}

In this section, we briefly illustrate the obtained results  using integral representations of finite groups.

Let $G$ be a finite group. Consider the integral group ring $\mathbb{Z}G$. It is well known that $\mathbb{Z}G$ is $1$-Gorenstein, and that a $\mathbb{Z}G$-module $M$ is Gorenstein projective if and only if the underlying abelian group of $M$ is free; for example, see \cite{Chen-Ren}. Such modules are sometimes called \emph{(generalized) $G$-lattices} \cite{BCK}. We denote $\mathbb{Z}G\mbox{-GProj}$ by $G\mbox{-Latt}$. The central problem in the integral representation theory is to study $G$-lattices, which motivates the study of singularity categories for general Gorenstein rings \cite{Buc}.

Let $p$ be a prime number, which is viewed as an element in $\mathbb{Z}G$. The element $p$ is regular and central. The corresponding matrix ring is
$$\Gamma=\begin{pmatrix} \mathbb{Z}G & \mathbb{Z}G\\
                         p\mathbb{Z}G & \mathbb{Z}G\end{pmatrix},$$
                         which  is $1$-Gorenstein by Proposition~\ref{prop:d-Gor}. Then the category $\mathbf{F}(\mathbb{Z}G; p)$ of module factorizations of $p$ over $\mathbb{Z}G$ is equivalent to $\Gamma\mbox{\rm -Mod}$. We mention that $\Gamma$ is isomorphic to the group ring of $G$ over the matrix ring $\begin{pmatrix} \mathbb{Z} & \mathbb{Z}\\
                         p\mathbb{Z} & \mathbb{Z}\end{pmatrix}$; see \cite[Subsection~6.7.1]{Zim}.

We observe that a module factorization $X=(X^0, X^1; d_X^0, d_X^1)$ of $p$ over $\mathbb{Z}G$ belongs to $\mathbf{GF}(\mathbb{Z}G; p)$ if and only if both components $X^i$ are $G$-lattices. Therefore, such $X$ might be called a \emph{lattice factorization} of $p$. We write $\mathbf{LF}(G; p)$ for $\mathbf{GF}(\mathbb{Z}G; p)$.   Similarly, we write $\mathbf{L}^0\mathbf{F}(G; p)$ for $\mathbf{G}^0\mathbf{F}(\mathbb{Z}G; p)$.

Denote by $\mathbb{F}_p$ the finite field of order $p$.  We identify the group algebra $\mathbb{F}_pG$ with the quotient ring $\mathbb{Z}G/(p)$. Since $\mathbb{F}_pG$ is Frobenius, any $\mathbb{F}_pG$-module is Gorenstein projective. Theorem~\ref{thm:cok0G} implies a triangle equivalence
$${\rm Cok}^0\colon \underline{\mathbf{L}^0\mathbf{F}}(G; p) \stackrel{\sim}\longrightarrow \mathbb{F}_pG\mbox{-}\underline{\rm Mod}.$$
Using this equivalence and Lemma~\ref{lem:compo}, the recollement in Proposition~\ref{prop:recGF} induces the following recollement.
\[
\xymatrix{
 \mathbb{F}_pG\mbox{-}\underline{\rm Mod}  \ar[rr] &&  \ar@/_1.5pc/[ll] \ar@/^1.5pc/[ll]|{\widetilde{\rm Cok}^0}  \underline{\mathbf{L}\mathbf{F}}(G; p)  \ar[rr]|{\rm pr^1} && G\mbox{-}\underline{\rm Latt} \ar@/_1.5pc/[ll]|{\theta^1}  \ar@/^1.5pc/[ll]|{\theta^0}
} \]
Consequently, if $p$ does not divide the order of $G$, $\mathbb{F}_pG$ is semisimple and the stable category $ \mathbb{F}_pG\mbox{-}\underline{\rm Mod}$ is trivial. In this situation, we have a triangle equivalence
$${\rm pr}^1\colon  \underline{\mathbf{L}\mathbf{F}}(G; p)\stackrel{\sim}\longrightarrow G\mbox{-}\underline{\rm Latt}. $$
We mention that the recollement above might play a role in the stratification theory  \cite{Bar, BBIKP} for $G\mbox{-}\underline{\rm Latt}$.

Denote by $C_n$ the cyclic group of order $n$. Consider the polynomial algebra $\mathbb{Z}[x]$ with one variable.  We identify the group ring $\mathbb{Z}C_n$ with the quotient ring $\mathbb{Z}[x]/(x^n-1)$. Therefore, we identify $C_n\mbox{-Latt}$ with $\mathbb{Z}[x]/(x^n-1)\mbox{-GProj}$.

Since the global dimension of $\mathbb{Z}[x]$ is two,  any Gorenstein projective $\mathbb{Z}[x]$-module is projective. Therefore, a module factorization $Y=(Y^0, Y^1; d_Y^0, d_Y^1)$ of $x^n-1$ over $\mathbb{Z}[x]$ belongs to $\mathbf{GF}(\mathbb{Z}[x]; x^n-1)$ if and only if both components $Y^i$ are projective $\mathbb{Z}[x]$-modules. Such $Y$ might be called a \emph{projective-module factorization} of $x^n-1$. We have $\mathbf{GF}(\mathbb{Z}[x]; x^n-1)=\mathbf{G}^0\mathbf{F}(\mathbb{Z}[x]; x^n-1)$, both of which are  better written as $\mathbf{PF}(\mathbb{Z}[x]; x^n-1)$.

Theorem~\ref{thm:cok0G} implies a triangle equivalence
$${\rm Cok}^0\colon \underline{\mathbf{PF}}(\mathbb{Z}[x]; x^n-1)\stackrel{\sim}\longrightarrow C_n\mbox{-}\underline{\rm Latt}.$$
If $p$ does not divide $n$, these categories are triangle equivalent to $\underline{\mathbf{L}\mathbf{F}}(C_n; p)$.

\vskip 5pt

\noindent {\bf Acknowledgements.}  \; The author thanks Professor Quanshui Wu very much for the reference \cite{CCKM} and his encouragement, and thanks Professor David Eisenbud for the reference \cite{EP}. He is very grateful to   Professor Hongxing Chen, Professor Bernhard Keller and Wenchao Wu for many helpful suggestions, and is indebted to Professor Li Liang for pointing out Theorem~\ref{thm:GP-2} and to Professor Peder Thompson for the references \cite{DM, BT}.

The project is supported by National Key R$\&$D Program of China (No. 2024YFA1013801) and  National Natural Science Foundation of China (No.s 12325101 and  12131015). Several days after the submission of this work to arXiv, the reference \cite{BFNS} appeared, whose main result might be viewed as a finite commutative version of Theorem~\ref{thm:cok0G}.

\vskip 10pt

 {\footnotesize \noindent Xiao-Wu Chen\\
 School of Mathematical Sciences, University of Science and Technology of China\\
 Hefei 230026, Anhui, PR China\\
 xwchen$\symbol{64}$mail.ustc.edu.cn}

\end{document}